\documentclass[final]{siamltex}
\usepackage{amssymb, amsmath}
\usepackage{float,epsfig}
\usepackage{algpseudocode}
\usepackage{algorithm}
\usepackage{enumitem}

\usepackage{tikz}
\newcommand{\Arrow}[1]{%
	\parbox{#1}{\tikz{\draw[->](0,0)--(#1,0);}}
}

\newcommand{\rsss}{\rotatebox[]{90}{$\boxminus$}\kern-0.7em{\mathrel{\raisebox{.2ex}{\Arrow{.35cm}}}}\!\!}

\newcommand{\csss}{\text{$\boxminus\kern-0.655em{\mathrel{\raisebox{-.2ex}{\rotatebox[]{-90}{\Arrow{.34cm}}}}}$\ }}

\newtheorem{remark}[theorem]{ Remark}
\newtheorem{exam}[theorem]{\bf Example}
\newtheorem{observation}[theorem]{ Observation}
\newcommand{\ba}{\begin{array}}
\newcommand{\ea}{\end{array}}

\newcommand{\be}{\begin{equation}}
\newcommand{\ee}{\end{equation}}
\newcommand{\beano}{\begin{eqnarray*}}
\newcommand{\eeano}{\end{eqnarray*}}

\def\C{{\mathbb C}}

\def\G{{\mathbb G}}

\def\lam{\lambda}

\def\diag{\mathrm{diag}}

\def\rank{\mathrm{rank}}

\def \nrank{\mathrm{nrk}}

\def \sp{\mathrm{Sp}}

\usepackage[us,12hr]{datetime}

\title{Vector spaces of linearizations for multivariable state-space systems
 }
\author{Avisek Bist \thanks{Department of Mathematics, Sikkim University, Sikkim-737102, India, ({\tt avisek.bista@gmail.com}).} \and  Namita Behera \thanks{Department of Mathematics, Sikkim University, Sikkim-737102, India, ({\tt nbehera@cus.ac.in}, niku.namita@gmail.com). } 
}

\begin{document}

\maketitle

\begin{abstract}
Consider a multivariable state space system and associated transfer function $G(\lam).$ The
aim of this paper is to define and analyze two vector spaces
of matrix pencils associated with the matrix $G(\lam)$ and show that almost all of these pencils are  linearizations of $G(\lam).$ We also construct symmetric/Hermitian linearizations of $G(\lam)$ when $G(\lam)$ is regular and symmetric/Hermitian.
\end{abstract}


	
\section{Introduction}
Consider a matrix polynomial $A(\lam) =\sum_{j=0}^{m} \lam^{j}A_j , \,\, A_j \in \C^{n \times n}$. Then  a matrix polynomial $A(\lam)$ is said to be regular if $\det(A(\lam)) \neq 0$ for some $\lam \in \C$. Linearization is a standard method for solving polynomial eigenvalue problems $A(\lam) x = 0$,  see~\cite{gohberg82,mmmm06,AV04,TDM} and references therein. Let $ A(\lam)$ be an $n\times n$ matrix polynomial (regular or singular) of degree $m.$ Then an $ mn\times mn$  matrix pencil $L(\lam) := X+ \lam Y$ is said to be a {\em linearization}~\cite{gohberg82} of  $A(\lam)$ if there are $mn\times mn$ unimodular matrix polynomials $U(\lam)$ (the determinant of $U(\lam),$ is a nonzero constant for all $\lam \in \mathbb{C}.$) and $V(\lam)$ such that $$  U(\lam) L(\lam) V(\lam) = \diag(I_{(m-1)n}, \,\, A(\lam)) $$   for all $\lam \in \C,$ where $I_k$ denotes the $k\times k$ identity matrix. 

% Linearization is a standard method for computing eigenvalues, eigenvectors, minimal bases and minimal indices of a matrix polynomial and has been studied extensively
% over the years,. 

There are two important class of vector space of linearizations of matrix polynomial which have been studied in \cite{mmmm06}. Recently, the idea presented in \cite{mmmm06} has been used in \cite{DA19} to study affine spaces of linearizations for LTI state-space system and associated transfer function.  In \cite{beh22, BehB22} Fiedler linearizations and recovery of eigenvectors of multivariable state-space systems have been studied extensively. One of the main objectives of this paper is to expand the arena in which to look for linearizations of multivariable state-space systems. This would provide us with more options to choose a linearization that possesses additional properties such as symmetric or Hermitian structure whenever $G(\lambda)$ is symmetric or Hermitian. To that end, we
construct two vector spaces of matrix pencils almost all of which are linearizations of $G(\lambda)$.

Now, consider a \textit{linear time invariant} (LTI) multivariable system $\Sigma$ given by

\begin{equation} \label{mulsss}
\begin{aligned}
	A\left( \frac{d}{dt}\right)x(t) &= B\left(\frac{d}{dt}\right)u(t) \\
	y(t)&= C\left( \frac{D}{dt}\right)x(t) + D\left( \frac{d}{dt}\right)u(t), \,\,\,\,\ t \geq 0,
\end{aligned}
\end{equation}
where $\frac{d}{dt}$ is the differential operator, $A(\lambda) \in \mathbb{C}[\lambda]^{n\times n}$ is regular, $B(\lambda) \in \mathbb{C}[\lambda]^{n\times m}$, $C(\lambda) \in \mathbb{C}[\lambda]^{p\times n}, D(\lambda) \in \mathbb{C}^{p\times m}, u(t): \mathbb{R}^+ \rightarrow \mathbb{R}^m$ is the output vector, $x(t): \mathbb{R}^+ \rightarrow \mathbb{R}^n$ is the state vector $y(t): \mathbb{R}^+ \rightarrow \mathbb{R}^p$ is the output vector.

Then the matrix polynomial $\mathcal{S}(\lambda)$ given by 

\begin{equation}
\mathcal{S}(\lambda) = \left[ \begin{array}{c|c} A(\lambda) & - B(\lambda) \\ \hline C(\lambda) & D(\lambda) \end{array} \right] \in \mathbb{C}[\lambda]^{(n+p)\times (n+m)} \label{rsm}  
\end{equation}
is called the \textit{Rosenbrock system matrix} or the \textit{Rosenbrock system polynomial} or simply called system matrix of the system $\Sigma$. The rational matrix $G(\lambda)$ given by
\begin{equation} \label{mtf}
 G(\lambda) = D(\lambda) + C(\lambda) A(\lambda)^{-1} B(\lambda) \in \mathbb{C}(\lambda)^{p\times m}     
\end{equation}
is called the transfer function of the system $\Sigma$.

The system $\Sigma$ is said to be in \textit{State Space Form} if it is given by

\begin{equation}\label{ssf} 
	\begin{aligned} 
		E\dot{x}(t) &= Ax(t) + Bu(t) \\
		y(t) &= Cx(t) + D(\lambda) u(t),
	\end{aligned}
\end{equation}
where $D(\lambda) \in \mathbb{C}[\lambda]^{p\times m}$ is a matrix polynomial and $A, E \in \mathbb{C}^{n\times n}$ with $E$ being non-singular and $B \in \mathbb{C}^{n \times m}, C \in \mathbb{C}^{p\times n}$ are constant matrices. 
%We denote the system defined above by $(E, A, B, C, A(\lambda))$. 
For computing zeros of a linear time-invariant system $\Sigma$ in state-space form given in (\ref{ssf}), Fiedler-like pencils and Rosenbrock linearization of the Rosenbrock system polynomial $\mathcal{S}(\lam)$ associated with $\Sigma$ have been introduced. Also, it has been shown that the Fiedler-like pencils are Rosenbrock linearizations of the system polynomial $\mathcal{S}(\lam)$, see \cite{rafinami1,rafinami2,rafinami3,behera20}. 

Further, for the Rosenbrock system matrix $\mathcal{S}(\lambda)$ given in (\ref{rsm}) associated to the  multivariable linear time invariant (LTI) state-space system $\Sigma$ given in (\ref{mulsss}),  recently, Fiedler-like pencils and its linearizations of $\mathcal{S}(\lambda)$ for both square and rectangular case have been introduced in \cite{beh22, BehB22, BBM23} to study the eigenvalues of $\mathcal{S}(\lambda)$.  
%Other than this in \cite{DopMQV20} the multivariable state-space system is considerd and a definition for local linearizations of rational matrices associated to $\Sigma$ is presented. This definition allows  to introduce matrix pencils associated to a rational matrix function that preserve its structure of zeros and poles in subsets of any algebraically closed field and also at infinity.

%For rational eigenvalue problems, in \cite{DopMQV20} the problem is considered as a multivariable state space system and a definition for local linearizations of rational matrices is presented. This definition allows  to introduce matrix pencils associated to a rational matrix function that preserve its structure of zeros and poles in subsets of any algebraically closed field and also at infinity.
Furthermore, in  \cite{AmpDMZ18, DMQD23, PerQ22}, different kind of linearizations of  the system matrix  $\mathcal{S}(\lambda)$ in  (\ref{rsm}) were studied. In \cite{AmpDMZ18}   strong linearizations of arbitrary rational matrices are discussed,  as well as characterizations of  linear matrix pencils for general rational matrices which are explicitly constructed from a minimal state-space realization of   a rational matrix. 
%For these linearizations of arbitrary rational matrices also the  recovery of eigenvectors has been analyzed when  $S(\lambda)$ is regular, as well as minimal bases and minimal indices, when  $S(\lambda)$ is singular, see \cite{AmpDMZ18,AmpDMZ21}. 
Here we assume for simplicity that $B, C$  are constant matrices in $\lambda$. However, all the results can be extended  to the case that $B, C$ depend on $\lambda$.
In \cite{DMQD23}  a new family of linearizations of rational matrices,  called block full rank linearizations has been studied. But it is not clear from \cite{AmpDMZ18, DMQD23} whether the pencils of  $\mathcal{S}(\lambda)$ which are studied in this paper are block permutationally similar to block minimal pencils.  Therefore,  we study explicitly the vector space  linearizations of system matrices/rational matrices associated to multivariable state-space systems.
In this paper, we construct two vector  spaces of matrix pencils almost all of which are linearizations of $\mathcal{S}(\lambda)$ or $G(\lambda)$ associated to the system  $\Sigma$.

%Recently, in \cite{PerQ22, AmpDMZ18, DMQD23, DMQD22, DopMQV20}, different linearizations of  the system polynomial $S(\lambda)$ in (\ref{rsm}) were studied. For these linearizations of arbitrary rational matrices the  recovery of eigenvectors has been analyzed when $S(\lambda)$ is regular, as well as minimal bases and minimal indices, when $S(\lambda)$ is singular, see \cite{AmpDMZ18,AmpDMZ21}. 

The paper is organized as follows.  In section~2 we recall the definition and some properties of matrix polynomial which we need throughout this paper. In section~3  we construct two vector spaces of matrix pencils almost all of which are linearizations of system matrix $\mathcal{S}(\lambda)$. In section~5 we discuss the vector space linerizations of $G(\lambda)$ given in (\ref{mtf}). In section~5 we define structure preserving linerizations (symmetric/ Hermitian) of $G(\lambda)$. In Section~6 we consider a linearization $L(\lambda)$ of $A(\lambda), D(\lambda)$ via a non-monomial basis
$\{\phi_{0}(\lambda), \phi_{1}(\lambda), \ldots, \phi_{m-1}(\lambda)\}$ of the space of scalar polynomials of degree at most $m-1, k-1$, respectively 
and then construct a corresponding linearization $\mathbb{L}(\lambda)$ of $G(\lambda)$ associated to the system  $\Sigma$.

{\bf Notation.} An $m\times n$ rational matrix function $G(\lam)$ is an $m\times n$ matrix whose entries are rational functions of the form $\frac{p(\lam)}{q(\lam)},$ where $p(\lam)$ and $q(\lam)$ are scalar polynomials in $\C[\lam].$  We denote  the $j$-th column of the $n \times n$ identity matrix $I_n$ by $e_j$ and the transpose of a matrix $A$ by $A^T.$ We denote the Kronecker product of matrices $A$ and $B$ by $A \otimes B.$

\section{Basic results}
The normal rank~\cite{rosenbrock70} of a rational matrix$ G(\lam) \in \C(\lam)^{m\times n}$ is denoted by $ \nrank(G)$ and is given by $\nrank(G) := \max_{\lam}\rank(G(\lam))$ where the maximum is taken over all $\lam \in \C$ which are not poles of the entries of $G(\lam).$ If $\nrank(G) =n = m$ then $G(\lam)$ is said to be regular, otherwise $G(\lam)$ is said to be singular. If $G(\lam)$ is regular then $\mu \in \C$ is said to be an {\em eigenvalue} of $G(\lam)$ if $ \rank(G(\mu)) < \nrank(G).$ See \cite{rafinami1} for more on eigenvalues and zeros of $G(\lambda)$.

Let $G(\lam)\in \mathbb{C}(\lam)^{p \times m}$ be a rational matrix with normal rank $k$ and let

\begin{equation*}\label{smform} 
\mathbf{SM}(G(\lam)) = \diag\left( \frac{\phi_{1}(\lam)}{\psi_{1}(\lam)}, \cdots, \frac{\phi_{k}(\lam)}{\psi_{k}(\lam)}, 0_{p-k, m-k}\right)    
\end{equation*} 
be the Smith-McMillan form~\cite{rosenbrock70} of  $G(\lam),$ where the scalar polynomials $\phi_{i}(\lam)$ and $\psi_{i}(\lam)$ are monic, are pairwise coprime and, $\phi_{i}(\lam)$ divides $\phi_{i+1}(\lam)$ and $\psi_{i+1}(\lam)$ divides $\psi_{i}(\lam),$ for $i= 1, 2, \ldots, k-1$. The polynomials $\phi_1(\lam), \ldots, \phi_k(\lam)$ and $\psi_1(\lam), \ldots, \psi_k(\lam)$ are called {\em invariant zero polynomials} and {\em invariant pole polynomials} of $G(\lam),$ respectively. Define
$\phi_{G}(\lam) := \prod _{j=1}^{k} \phi_{j}(\lam) \,\,\, \mbox{ and } \,\,\, \psi_{G}(\lam) := \prod _{j=1}^{k} \psi_{j}(\lam).$ A complex number $\lam$ is said to be a  zero of $G(\lam)$ if $ \phi_G(\lam) =0$ and a complex number $ \lam$ is said to be a pole of $G(\lam)$ if $\psi_G(\lam) =0.$ The {\bf spectrum} $\sp(G)$ of $G(\lam)$ is given by  $\sp(G) :=\{\lam \in \C : \phi_G(\lam) = 0\}.$ That is $\sp(G)$ is the set of zeros of $G(\lam),$ see \cite{rafinami1}.

A complex number $\lambda$ is said to be an eigenvalue of the system $\mathcal{S}(\lambda)$ if $\rank(\mathcal{S}(\lambda)) < \nrank(\mathcal{S})$. Equivalently, $\lambda \in \mathbb{C}$ is an eigenvalue of $\mathcal{S}(\lambda) \Leftrightarrow \lambda$ is a root of the zero polynomial $\phi_s (\lambda)$ of $\mathcal{S}(\lambda)$ given by $\phi_s (\lambda) = \phi_1 (\lambda) \phi_2 (\lambda) \dots \phi_k (\lambda)$ where $\phi_i (\lambda)$ are the invariant polynomials of $\mathcal{S}(\lambda)$ for $i = 1, 2, \dots, k$. An eigenvalue $\lambda$ of $\mathcal{S}(\lambda)$ is called an \textbf{invariant zero} of the system $\Sigma$. The set of eigenvalues of $\mathcal{S}(\lambda)$ is denoted by $\sp(\mathcal{S})$. The zeros of $\phi_G (\lambda)$ are called the \textbf{transmission zeroes} of $\Sigma$. The zeroes of $\psi_G (\lambda)$ are called the \textbf{poles} of the system $\Sigma$.

For a matrix polynomial $A(\lambda) = \sum_{j=0}^{m} {\lambda^j A_j}\in\mathbb{C}[\lambda]^{n\times n}$, the pencils $C_1 (\lambda) = \lambda X_1 + Y_1 $ and $C_2 (\lambda) = \lambda X_2 + Y_2 $ respectively, are called the first and second companion pencils of $A(\lambda)$ where $ X_1 = X_2 = \diag(A_m, I_{(m-1)n})$ and,
	$$Y_1 = \left[ \begin{array}{ccccc} A_{m-1} & A_{m-2} & \cdots & A_1 & A_0 \\ -I_n & 0 & \cdots & 0 & 0 \\ \vdots & \vdots & \ddots & \vdots & \vdots \\ 0 & 0 & \cdots & -I_n & 0 \end{array} \right] \,\,\,\,\,\,\, Y_2 = \left[ \begin{array}{ccccc} A_{m-1} & -I_n & \cdots & 0 & 0 \\ A_{m-2} & 0 & \cdots & 0 & 0 \\ \vdots & \vdots & \ddots & \vdots \\ A_1 & 0 & \cdots & 0 & -I_n \\ A_0 & 0 & \cdots & 0 & 0 \end{array}    \right]$$

Generalizing the companion pencils of $A(\lambda)$, two vector spaces $\mathcal{L}_1$ and $ \mathcal{L}_2 $ of $ mn\times mn $ matrix pencils $ L(\lambda) = \lambda X + Y $ are defined as, \cite{mmmm06}
	\begin{align}
		\mathcal{L}_1(A) &= \left\{ L(\lambda) : L(\lambda) \left( \Lambda \otimes I_n \right) = v \otimes P(\lambda), v(\lambda) \in \mathbb{C}^m \right\} \\ 	      
		\mathcal{L}_2(A) &= \left\{ L(\lambda) : \left( \Lambda^T \otimes I_n \right) L(\lambda) = v \otimes P(\lambda), v(\lambda) \in \mathbb{C}^m \right\}
	\end{align}
	
For convenience, we write $(L(\lambda), v) \in \mathcal{L}_1 (A)$ to mean that $L(\lambda) \in \mathcal{L}_1 (A)$ with right ansatz vector $v$. Similarly, we write $(L(\lambda), w) \in \mathcal{L}_2 (A)$ to mean that $L(\lambda) \in \mathcal{L}_2 (A)$ with left ansatz vector $w$. For any $n \times n$  matrix polynomial $A$ of degree $m$, the double ansatz space of $A$ is given by $ \mathcal{DL}(A) := \mathcal{L}_{1}(A) \cap \mathcal{L}_{2}(A)$, see \cite{mmmm06}.
	
\begin{definition} \cite{mmmm06}
Let $X$ and $Y$ be block $d\times d$ matrices

$$ X =\left[\begin{array}{ccc}
           X_{11} & \cdots & X_{1d} \\
           \vdots & \ddots & \vdots \\
           X_{d1}& \cdots& X_{dd}\\
      \end{array}\right],  \,\,\,\,\, 
  Y=\left[\begin{array}{ccc}
           Y_{11} & \cdots & Y_{d1} \\
           \vdots & \ddots & \vdots \\
           Y_{d1}& \cdots & Y_{dd}\\
      \end{array}\right]$$
with blocks $X_{ij},Y_{ij}\in\,\mathbb{C}^{n\times n}$. Then the column shifted sum  of $X$ and $Y$ is defined to be

$$X\boxplus Y:=\left[\begin{array}{cccc}
                   X_{11} & \cdots & X_{1d}& 0 \\
                   \vdots &  &  & \vdots\\
                   X_{d1}& \cdots & X_{dd}& 0\\
                \end{array}\right]  + 
                \left[\begin{array}{cccc}
                   0 & Y_{11} & \cdots & Y_{d1} \\
                   \vdots &  & & \vdots \\
                   0 & Y_{d1} & \cdots  & Y_{dd}\\
                   \end{array}\right],$$
and the \textbf{row shifted sum} of $X$ and $Y$ is defined to be

$$X\widehat {\boxplus }Y:=\left[\begin{array}{cccc}
                              X_{11} & \cdots & X_{1d}\\
                              \vdots &  &  \vdots \\
                              X_{d1}& \cdots & X_{dd}\\
                              0 & \cdots & 0  \\
                          \end{array}
                          \right]  + 
                          \left[\begin{array}{cccc}
                             0  & \cdots & 0 \\
                             Y_{11} & \cdots & Y_{d1} \\
                             \vdots &  & \vdots \\
                             Y_{d1}& \cdots & Y_{dd}\\
                         \end{array}\right],$$  
where the zero blocks are of size $n\times n$.    
\end{definition}

\section{Vector Space of Linearizations of Rosenbrock System Matrix}

Consider the \emph{Multivariable linear time invariant (LTI) state-space system} $\Sigma$ given by

\begin{equation}
    \begin{aligned}
       A\left(\frac{d}{dt}\right) x(t) &= B u(t), \\
       y(t) &= C x(t) + D\left(\frac{d}{dt}\right) u(t),
    \end{aligned} \label{hssf01}
\end{equation}
where $A(\lambda) = \sum_{j=0}^{m}\lambda^{j}A_j \in \mathbb{C}[\lambda]^{n \times n}$ is regular matrix polynomial of degree $m$ and $ D(\lambda) = \sum_{j=0}^{k}\lambda^{j}D_j \in \mathbb{C}[\lambda]^{r \times r}$ is a  matrix polynomial of degree $k$,  $ C \in \mathbb{C}^{r \times n}, B \in \mathbb{C}^{n \times r}$. The associated system matrix is given by

\begin{equation}\label{sysmatrix01}
\mathcal{S}(\lambda) = \left[\begin{array}{c|c}
                            A(\lambda) & -B \\ \hline
                            C & D (\lambda) \\
                       \end{array}\right] 
                       \in \mathbb{C}[\lambda]^{(n+r)\times (n+r)} 
\end{equation}
and the transfer function of $\Sigma$ is given by

\begin{equation}\label{tfunction01}
G(\lambda) = C A(\lambda)^{-1} B + D(\lambda) \in \mathbb{C}(\lambda)^{r \times r}. 
\end{equation}

The eigenvalue problem associated with the system matrix $\mathcal{S}(\lambda)$ given in (\ref{sysmatrix01}) is to find a vector $u\in\mathbb{C}^{n+r}$ such that $\mathcal{S}(\lambda)u=0$. If we write $u=\left[\begin{array}{c} x \\ \hline y\end{array}\right]$ where $x\in\mathbb{C}^{n}$ and $y\in\mathbb{C}^{r}$ then, it is easy to see that, 
\begin{align*}
    \mathcal{S}(\lambda)u =0
   \iff &\left[\begin{array}{c|c}
                            A(\lambda) & -B \\ \hline
                            C & D (\lambda) \\
                       \end{array}\right]
                       \left[\begin{array}{c}
                         x \\ \hline y
                       \end{array}\right]=0 \\
    \iff&\left[\begin{array}{c} A(\lambda)x-By \\ Cx+D(\lambda)y\end{array}\right]=0\\
    \iff &  A(\lambda)x-By=0 \,\ \text{and}\,\ Cx+D(\lambda)y=0
\end{align*}

Consider $G(\lambda)$ given in (\ref{tfunction01}). Then first companion linearization of $\mathcal{S}(\lambda)$ or $G(\lambda)$ are given by 

\begin{eqnarray*}
\mathcal{C}_{1}(\lambda) &= & \lambda \left[
                         \begin{array}{cccc|cccc}
                           A_{m} &  &  &  & & & & \\
                            & I_{n} &  &  & & & & \\
                            &  & \ddots &  & & & & \\
                            &  &  & I_{n} &  & & &\\
                            \hline
                            &  &  &  & D_k & & & \\
                            &  &  &  &  & I_r& & \\
                            &  &  &  &  & & \ddots & \\
                            &  &  &  &  & & & I_r \\
                         \end{array}
                       \right] \nonumber\\
                       && \quad+ \left[
                \begin{array}{cccc|cccc}
                  A_{m-1} & A_{m-2}& \cdots & A_{0}& 0& \cdots & 0& -B  \\
                   -I_{n} & 0 & \cdots & 0 & 0 & \cdots & & 0\\
                   & \ddots &  & \vdots  & \vdots & & & \vdots\\
                   &  & -I_{n} & 0 & 0 &  \cdots & & 0 \\
                   \hline
                 0 & \cdots & 0 & C &  D_{k-1} & D_{k-2} &\cdots & D_0 \\
                  & 0 & \cdots & 0 &  -I_r & 0 & \cdots & 0 \\
                  &  & \ddots & \vdots &   & \ddots & & \vdots\\
                  &  &  & 0 &   & & -I_r & 0 \\
                \end{array}
              \right]
\end{eqnarray*}

Now, for all $x \in \mathbb{C}^{n}$ and $y\in\mathbb{C}^{r}$, we have,

\begin{equation}\label{baseeqn01}
\mathcal{C}_{1}(\lambda) 
                        \left[ \begin{array}{c}
                         \Lambda_{m-1} \otimes x \\
                         \Lambda_{k-1} \otimes y
                         \end{array}\right]
                = \left[\begin{array}{c} A(\lambda)x-By \\0\\ \vdots\\ 0 \\\hline Cx+D(\lambda)y\\ 0 \\ \vdots\\ 0 \end{array}\right],
\end{equation}
where $\Lambda_{p-1}^{T}=\left[\begin{array}{cccc} \lambda^{p-1} & \ldots & \lambda & 1 \end{array}\right]^{T}$. Thus, any solution of (\ref{baseeqn01}) leads to the solution of the problem $\mathcal{S}(\lambda)u=0,$ where 
$u=\left[\begin{array}{c} x \\ y \end{array}\right]\in\mathbb{C}^{n+r}$.
Let us suppose $r<n$ and define

$$I_{r\times n}=\left[\begin{array}{ccccccc} 
                1 & 0 & \cdots & 0 & 0 & \cdots & 0 \\
                0 & 1 & \cdots & 0 & 0 & \cdots & 0  \\
                \vdots& \vdots & \ddots & \vdots & \vdots & \ddots & \vdots \\
                0 & 0 & \cdots & 1 & 0 & \cdots & 0
                \end{array}\right]$$
that is, $I_{r\times n}$ is the $r\times r$ identity matrix augmented with $n-r$ zero columns. Also, let us denote by $e^{p}_{q}$, the $q$-th column of the identity matrix $I_p$. Then, (\ref{baseeqn01}) is equivalent to  the identity 
$$\mathcal{C}_{1}(\lambda)\left[\begin{array}{c}
                         \Lambda_{m-1} \otimes I_n\\
                         \Lambda_{k-1} \otimes I_{r\times n}
                         \end{array}\right] 
                    = \left[\begin{array}{c}
                         A(\lambda)-BI_{r\times n} \\
                         \vdots \\
                         0 \\ \hline
                        C+D(\lambda)I_{r\times n}  \\
                         0 \\
                         \vdots \\
                         0 \\
                         \end{array}\right]
                    = \left[\begin{array}{c}
                         e^{m}_{1}\otimes \left(A(\lambda)-BI_{r\times n}\right) \\
                         \hline
                         e^{k}_{1} \otimes \left(C+D(\lambda)I_{r\times n}\right) \\
                          \end{array}
              \right].$$
If we generalize this to any arbitrary pencil then we consider the set of pencils $\mathbb{L}(\lambda) := \lambda \mathbb{X} + \mathbb{Y}$ satisfying the property

{\scriptsize
\begin{equation} \label{ransv}
\mathbb{L}(\lambda) \left[\begin{array}{c}
                         \Lambda_{m-1} \otimes I_{n}\\
                         \Lambda_{k-1} \otimes I_{r\times n}
                         \end{array}\right] 
                    = \mathbb{L}(\lambda)\left[\begin{array}{c} 
                         \lambda^{m-1} I_{n} \\
                         \vdots \\
                          \lambda I_{n} \\
                          I_{n} \\ \hline
                          \lambda^{k-1} I_{r\times n} \\
                          \vdots \\
                          \lambda I_{r\times n} \\
                          I_{r\times n}
                         \end{array}\right]
                    = \left[\begin{array}{c}
                         v_{1}\left(A(\lambda)-BI_{r\times n}\right)\\
                         \vdots \\
                         v_{m}\left(A(\lambda)-BI_{r\times n}\right)\\
                         \hline 
                        w_{1}\left(C+D(\lambda)I_{r\times n}\right) \\
                         \vdots \\
                         w_{k}\left(C+D(\lambda)I_{r\times n}\right)\\
                         \end{array}\right] 
                    = \left[\begin{array}{c}
                      v\otimes\left(A(\lambda)-BI_{r\times n}\right)\\\hline
                      w \otimes \left(C+D(\lambda)I_{r\times n}\right)
                      \end{array} \right]     
\end{equation}} 
for some vector $\left[
                         \begin{array}{c}
                         v\\
                         \hline
                          w 
                         \end{array}\right]\in \mathbb{C}^{m+k} $ 
such that $v \in \mathbb{C}^m $ and $w \in \mathbb{C}^k $. This set of pencils will be denoted by $\mathbb{L}_{1}(\mathcal{S})$ as a reminder that it generalizes the first companion form of $\mathcal{S}$.

Define  $$\Gamma_{\mathcal{S}} = \left\{\left[\begin{array}{c}
                      v\otimes\left(A(\lambda)-BI_{r\times n}\right)\\\hline
                      w \otimes \left(C+D(\lambda)I_{r\times n}\right)
                      \end{array} \right]  : v\in\mathbb{C}^{m},w \in \mathbb{C}^{k}\right\}.$$

\begin{lemma}
$\Gamma_{\mathcal{S}}$ is a vector space under the Kronecker product.
\end{lemma}

Notice that from the properties of Kronecker product it is easy to see that $\Gamma_{\mathcal{S}}$ is a vector space isomorphic to $\mathbb{C}^{m+k}$.
Now, we have the following definition.

\begin{definition}
Consider the classes of pencils

{\small
\begin{align*}
& \mathbb{L}_{1}(\mathcal{S}) := \left\{ \mathbb{L}(\lam) =\left[
                                        \begin{array}{c|c}
                                          L(\lambda) & -v e_{k}^{T} \otimes B  \\
                                          \hline
                                           we_{m}^{T} \otimes C & K(\lambda) \\
                                        \end{array}
                                      \right] :  
(L(\lambda), v) \in \mathcal{L}_{1}(A),  \,\,  (K(\lambda), w) \in \mathcal{L}_{1}(D)
       \right\}.
\end{align*} }
\end{definition}

Note that if $\mathbb{L}(\lambda) \in \mathbb{L}_{1}(\mathcal{S})$,
then 
$\mathbb{L}(\lam).\left[\begin{array}{c}
                         \Lambda_{m-1} \otimes I_{n}\\
                         \Lambda_{k-1} \otimes I_{r\times n}
                         \end{array}\right] \in \Gamma_{\mathcal{S}}. $

We will sometimes use the phrase "$\mathbb{L}(\lambda)$ satisfies the right ansatz with vector $\left[\begin{array}{c}
                         v\\
                         w 
                         \end{array}\right] $ or we say
$\left[\begin{array}{c}
                         v\\
                         w 
                         \end{array}\right] $ is the right ansatz vector for $\mathbb{L}(\lambda)$" when $\mathbb{L}(\lambda) \in \mathbb{L}_{1}(\mathcal{S})$  and the vector $\left[\begin{array}{c}
                         v\\
                         w 
                         \end{array}\right] $ in (\ref{ransv}).

\begin{lemma}
    $\mathbb{L}_{1}(\mathcal{S})$ is a vectorspace isomorphic to $\mathcal{L}_{1}(A)\times \mathcal{L}_{1}(D)$.
\end{lemma}
\begin{proof}
    Define a function $\phi:\mathbb{L}_{1}(\mathcal{S}\to\mathcal{L}_{1}(A)\times\mathcal{L}_{1}(D)$ such that
    $$\phi\left(\left[\begin{array}{c|c}
                    L(\lambda) & -v e_{k}^{T} \otimes B  \\\hline
                    we_{m}^{T} \otimes C & K(\lambda) \\
                    \end{array}\right]\right)       =\left(\left(L(\lambda,v)\right),\left(K(\lambda),w\right)\right).$$ Then $\phi$ is linear, one one and onto. 
\end{proof}

\begin{definition}(Block Column Shifted Sum)
Let $X=\left[\begin{array}{c|c}X_{11} & X_{12} \\\hline X_{21} & X_{22} \end{array}\right]$ and $Y=\left[\begin{array}{c|c}Y_{11} & Y_{12} \\\hline Y_{21} & Y_{22} \end{array}\right]$ be two $(mn+rk)\times(mn+rk)$ matrices such that $X_{11},Y_{11}\in\mathbb{C}^{mn\times mn}$ and $X_{22},Y_{22}\in\mathbb{C}^{rk\times rk}$. Then, we define the block column shifted sum of $X$ and $Y$ as,

$$X \dashv\vdash Y=\left[\begin{array}{c|c} X_{11}\boxplus Y_{11} & X_{12}\boxplus Y_{12} \\ \hline X_{21}\boxplus Y_{21} & X_{22}\boxplus Y_{22}\end{array}\right]$$
where $\boxplus$ denotes the column shifted sum defined in \cite{mmmm06}.
\end{definition}

\begin{lemma}
Let
$$\mathbb{L}(\lambda)=\left[\begin{array}{c|c}
                             L(\lambda) & -v e_{k}^{T} \otimes B  \\ 
                             \hline
                            we_{m}^{T} \otimes C & K(\lambda) \\ \end{array}\right]
                        =\left[\begin{array}{c|c}\lambda X+Y &ve^{k}_{k}\otimes B \\ \hline
                           we^{m}_{m}\otimes C & \lambda P+Q\end{array} \right] \\
                           $$$$=\lambda\left[\begin{array}{c|c}X &0 \\\hline
                           0 & P\end{array}\right]+\left[\begin{array}{c|c}Y &ve^{k}_{k}\otimes B \\\hline
                           we^{m}_{m}\otimes C & Q\end{array}\right]=\lambda\mathbb{X}+\mathbb{Y}$$
Then,\\ 
{\small $\mathbb{L}(\lambda) \in \mathbb{L}_{1}(\mathcal{S}) \iff \mathbb{X}\dashv\vdash\mathbb{Y}
=\left[\begin{array}{c|c} 
   v\otimes \left[\begin{array}{ccc}A_{m}&\cdots&A_0\end{array}\right] & v\otimes \left[\begin{array}{cccc} 0 & \cdots & 0 & -B\end{array}\right] \\ 
\hline 
w\otimes \left[\begin{array}{cccc} 0 & \cdots & 0 & C \end{array}\right] & w\otimes\left[\begin{array}{ccc} D_{k} &\cdots & D_{0}
\end{array}\right] 
\end{array}\right]$}
\end{lemma}

\begin{proof}
Straightforward calculation.
\end{proof}

The following theorem characterizes the pencils in $\mathbb{L}_{1}(\mathcal{S})$.

\begin{theorem}[Characterizations of Pencils in $\mathbb{L}_{1}(\mathcal{S})$] \label{cl1s}
Let $\mathcal{S}(\lambda)$ be an $(n+r)\times (n + r)$ system matrix given in (\ref{sysmatrix01}), and $\left[
                                  \begin{array}{c}
                                    v \\
                                    w \\
                                  \end{array}
                                \right]
 \in\, \mathbb{C}^{m+k}$ be any vector. Then the set
 of pencils in $\mathbb{L}_{1}(G)$ with right ansatz vector 
 $\left[
        \begin{array}{c}
        v \\
       w \\
     \end{array}
     \right]
 $ consists of all $\mathbb{L}(\lambda)= \lambda  \mathbb{X}+\mathbb{Y}$ such that
\[
\mathbb{X} = \bordermatrix{ & n  & (m-1)n & r & (k-1)r\cr
                   &  v\otimes A_{m} & W_1   & 0 & 0\cr
                   & 0 & 0 & w\otimes D_{k} & W_{2} \cr
                   }
\]
and \[
\mathbb{Y} = 
 \bordermatrix{ &  (m-1)n & n &  (k-1)r & r\cr
          &  v \otimes [A_{m-1} \cdots A_{1} ] -W_1 & v \otimes A_{0}  & 0 & 
          -ve_{k}^{T}\otimes B  \cr
        & 0 & we_{m}^{T}\otimes C  & -W_{2}+ w \otimes [D_{k-1} \cdots D_{1}] &  w \otimes D_{0}\cr
                   }
\]
with $W_1\in\, \mathbb{F}^{(nm)\times n(m-1)}$, $W_{2}\in \mathbb{F}^{(rk)\times r(k-1)}$ chosen arbitrarily.    
\end{theorem}

\begin{proof}
Consider a map {\small $\mathcal{F}:\mathbb{L}_{1}(\mathcal{S})\to \Gamma_{\mathcal{S}}$} defined as {\small $\mathbb{L}(\lambda)\longmapsto \mathbb{L}(\lambda)\left[\begin{array}{c} \Lambda_{m-1} \otimes I_{n}\\\Lambda_{k-1} \otimes I_{r\times n}\end{array}\right]$}. Then, for any two pencils $\mathbb{L}_{1}(\lambda)$ and $\mathbb{L}_{2}(\lambda)$ in $\mathbb{L}_{1}(\mathcal{S})$, we have 
\begin{align*}
\mathcal{F} \left(\alpha \mathbb{L}_{1}(\lambda)+\mathbb{L}_{2}(\lambda)\right) & = (\alpha \mathbb{L}_{1}(\lambda)+\mathbb{L}_{2}(\lambda))
     \left[ \begin{array}{c}\Lambda_{m-1} \otimes I_{n} \\
                         \Lambda_{k-1} \otimes I_{r\times n}
                         \end{array}\right]\\
& = \alpha \mathbb{L}_{1}(\lambda) \left[
                         \begin{array}{c}
                     \Lambda_{m-1} \otimes I_{n} \\
                         \Lambda_{k-1} \otimes I_{r\times n}
                         \end{array}
              \right]+ \mathbb{L}_{2}(\lambda)\left[
                         \begin{array}{c}
                     \left(\Lambda_{m-1} \otimes I_{n} \right) \\
                         \Lambda_{k-1} \otimes I_{r\times n}
                         \end{array}\right] \\
&= \alpha \mathcal{F}(\mathbb{L}_{1}(\lambda)) +  \mathcal{F}(\mathbb{L}_{2}(\lambda))
\end{align*}
This shows that $\mathcal{F}$ is linear.
Next, let $\left[\begin{array}{c} v\otimes \left(A(\lambda)-BI_{r\times n}\right) \\\hline w\otimes\left(C+D(\lambda)I_{r\times n}\right)\end{array}\right]$ be an arbitrary element of $\Gamma_{\mathcal{S}}$. Construct $\mathbb{X}$ and $\mathbb{Y}$ as follows:
$$\mathbb{X}=\left[\begin{array}{cc|cc} 
              v\otimes A_{m} & 0 & 0 & 0 \\ \hline
              0 & 0 & w\otimes D_{k} & 0
               \end{array}\right]$$
and 
$$\mathbb{Y}=\left[\begin{array}{c|c}
             v\otimes\left[\begin{array}{cccc} A_{m-1}&\cdots & A_{1} & A_{0} \end{array}\right] & v \otimes \left[\begin{array}{cccc} 0 & \cdots & 0 & -B \end{array}\right] \\ \hline
             w\otimes \left[\begin{array}{cccc} 0 & \cdots & 0 & C \end{array}\right] & w\otimes \left[\begin{array}{cccc}D_{k-1} & \cdots & D_{1} & D_{0} \end{array}\right]\end{array}\right]$$
Then, we see that
\begin{align*}
\mathbb{X}\dashv\vdash \mathbb{Y} 
&=  \left[\begin{array}{ccc|ccc}  
       v \otimes A_{m} & 0 & 0 & 0 & 0 & 0 \\\hline
       0 & 0 & 0 &  w \otimes D_{k} & 0  & 0 \\
    \end{array}\right] \\
    &\hspace{2mm}+ 
    \left[\begin{array}{ccc|ccc}
       0 &  v \otimes [A_{m-1} \cdots A_{1}] & v \otimes A_0 & 0  & 0& -ve_{k}^{T} \otimes B  \\\hline
       0 & 0 & w e_{m}^{T} \otimes C & 0 & w \otimes [D_{k-1} \cdots D_{1}] & w \otimes D_0\\
    \end{array}\right] \\            
 & = \left[\begin{array}{ccc|ccc}
       v \otimes A_{m} &  v \otimes [A_{m-1} \cdots A_{1}] &  v \otimes A_0 & 0 & 0 & -ve_{k}^{T} \otimes B  \\\hline
       0 & 0 & w e_{m}^{T} \otimes C & w \otimes D_{k} & w \otimes [D_{k-1} \cdots D_{1}] & w \otimes D_0\\
     \end{array}\right] \\
 &= \left[\begin{array}{cc|cc}
       v \otimes [A_{m} \,\, A_{m-1} \cdots A_{1}] &  v \otimes A_0 & 0 & -ve_{k}^{T} \otimes B \\\hline
       0 & w e_{m}^{T} \otimes C & w \otimes [D_{k} \,\, D_{k-1} \cdots D_{1}] & w\otimes D_0 \\
    \end{array}\right] \\
    & =\left[\begin{array}{c|c}
             v\otimes\left[\begin{array}{ccccc} A_{m} & A_{m-1}&\ldots & A_{1} & A_{0} \end{array}\right] & v \otimes \left[\begin{array}{cccc} 0 & \cdots & 0 & -e_{k}^{T}\otimes B \end{array}\right] \\ \hline
             w\otimes \left[\begin{array}{cccc}0 & \cdots & 0 & e_{m}^{T}\otimes C \end{array}\right] & w\otimes \left[\begin{array}{ccccc} D_{k} & D_{k-1} & \ldots & D_{1} & D_{0} \end{array}\right]\end{array}\right] 
             \end{align*}
This shows that $\lambda\mathbb{X}+\mathbb{Y}$ is an $\mathcal{F}$-perimage of $\left[\begin{array}{c} v\otimes \left(A(\lambda)-BI_{r\times n}\right) \\\hline w\otimes\left(C+D(\lambda)I_{r\times n}\right)\end{array}\right]$ and hence $\mathcal{F}$ is onto.

Next, we find the kernel of $\mathcal{F}$. Let 
$$\mathbb{L}(\lambda)=\lambda\mathbb{X}+\mathbb{Y}
                     =\left[\begin{array}{c|c}
                      \lambda X +Y & -ve_{k}^{T}\otimes B \\ \hline
                       we_{m}^{T}\otimes C & \lambda P + Q \end{array}\right] \in \ker \mathcal{F}.
                       $$
Then $\mathbb{X} \dashv\vdash\mathbb{Y} =0$ implies that
\begin{align*}
&\left[\begin{array}{c|c}
        X  & 0 \\ \hline
        0 &  P  \end{array}\right] \dashv\vdash
        \left[\begin{array}{c|c}
        Y & -ve_{k}^{T}\otimes B \\ \hline
        we_{m}^{T}\otimes C &  Q \end{array}\right]=0 \\
 \implies & \left[\begin{array}{c|c}
            X\boxplus Y  &  0 \boxplus -ve_{k}^{T}\otimes B \\ 
            \hline
             0 \boxplus we_{m}^{T}\otimes C  & P \boxplus Q \end{array}\right] =0 \\
  \implies & X\boxplus Y  = 0,\quad P \boxplus Q= 0, \quad -v e_{k}^{T}\otimes B=0\quad \text{and}  \quad we_{m}^{T}\otimes C=0 \\
 \implies & X =  \left[\begin{array}{cc}
    0 & - W_1 \\
    \end{array}\right]\quad \text{and}\quad Y =  \left[\begin{array}{cc}
     W_1 & 0 \\
    \end{array}\right]  \\
     &  P =  \left[\begin{array}{cc}
    0 & - W_2 \\
    \end{array}\right]\quad \text{and}\quad Q =  \left[\begin{array}{cc}
     W_2 & 0 \\
    \end{array}\right], 
\end{align*}
where $W_{1}\in\mathbb{C}^{mn\times (m-1)n}$ and $W_{2}\in\mathbb{C}^{rk\times (k-1)r}$. Therefore, we have 
\[
\mathbb{X} = \bordermatrix{ & n  & (m-1)n \cr
                   &  v\otimes A_{m} & W_1   & 0 & 0\cr
                   & 0 & 0 & w\otimes D_{k} & W_{2} \cr
                   }
\]
and \[
\mathbb{Y} = 
 \bordermatrix{ &  (m-1)n & n \cr
          &  v \otimes [A_{m-1} \cdots A_{1} ] -W_1 & v \otimes A_{0}  & 0 & 
          -ve_{k}^{T}\otimes B  \cr
        & 0 & we_{m}^{T}\otimes C  & -W_{2}+ w \otimes [D_{k-1} \cdots D_{1}] &  w \otimes D_{0}\cr
                   }
\]
Hence proved.
\end{proof}

\begin{corollary}
 $\dim \mathbb{L}_1{S} = \dim \mathcal{L}_{1}(A) + \dim \mathcal{L}_{1}(D). $
\end{corollary}

\begin{proof}
From the above theorem we have  
\begin{align*}
\dim \mathbb{L}_1(S) & = \dim \Gamma_{\mathcal{S}} + \dim \ker \mathcal{F} \\
   & = m+k + (m-1)n . mn + (k-1)r. rk \\
   & = m+ m (m-1)n^2 + k+ k (k-1)r^2 \\
   & =  \dim \mathcal{L}_{1}(A) + \dim \mathcal{L}_{1}(D).
\end{align*}   
\end{proof}

\begin{theorem}
Define $f: \C^{r} \rightarrow \C^{n+r}$ and $g : \C^{r} \rightarrow \C^{n+r}$ by

$$f(x) := \left[\begin{array}{c}
               A(\lambda)^{-1}Bx \\
               x \\
          \end{array}\right]\quad \mbox{and} \quad
g(x) := \left[\begin{array}{c}
           (-CA(\lambda)^{-1})^{T}x \\
            x \\
        \end{array}\right].$$ 
Then the maps
$f : \mathcal{N}_{r}(G(\lam_0))\rightarrow \mathcal{N}_{r}(\mathcal{S}(\lam_0))$  and $g  : \mathcal{N}_l(G(\lam_0))\rightarrow \mathcal{N}_l(\mathcal{S}(\lam_0))$ are isomorphisms. 
\end{theorem}

\begin{proof}
Let $x \in \mathcal{N}_{r}(G(\lam_0)).$ Then $G(\lam_0)x = 0$. Now
 
$$\mathcal{S}(\lambda_0)\left[\begin{array}{c}
                            A(\lambda_0)^{-1}Bx \\
                            x \\
                        \end{array}\right] 
        = \left[\begin{array}{c|c}
             A(\lam_0) & -B \\\hline
             C & D(\lam_0)  \\
          \end{array}\right]
          \left[\begin{array}{c}
             A(\lambda_0)^{-1}Bx \\
             x \\
          \end{array}\right]  
        = \left[\begin{array}{c}
             Bx-Bx \\
             G(\lam_0)x    \\
          \end{array}\right] 
        = \left[\begin{array}{c}
             0 \\
             0 \\
          \end{array}\right]$$
which shows that $f(x)  \in \mathcal{N}_{r}(\mathcal{S}(\lam_0))$ and
 $f: \mathcal{N}_{r}(G(\lam_0))\rightarrow \mathcal{N}_{r}(\mathcal{S}(\lam_0)).$ Obviously $f$ is injective.
 On the other hand, if $\left[\begin{matrix} u^T, & v^T \end{matrix} \right]^T \in \mathcal{N}_{r}(\mathcal{S}(\lam_0))$ then it is easy to see that $ G(\lam_0) v = 0 $ and $ u = A(\lam_0)^{-1} B v.$ Hence $\left[\begin{matrix} u^T, & v^T \end{matrix} \right]^T = f(v) $ and $ v \in \mathcal{N}_{r}(G(\lam_0)).$ This shows that $f: \mathcal{N}_{r}(G(\lam_0))\rightarrow \mathcal{N}_{r}(\mathcal{S}(\lam_0))$ is an isomorphism.  Similarly, the map $g  : \mathcal{N}_l(G(\lam_0))\rightarrow \mathcal{N}_l(\mathcal{S}(\lam_0))$ is an  isomorphism.
\end{proof}

%Since $f: \mathcal{N}_{r}(G(\lam_0)) \rightarrow \mathcal{N}_{r}(\mathcal{S}(\lam_0)) $ and $g : \mathcal{N}_{l}(G(\lam_0)) \rightarrow \mathcal{N}_{l}(\mathcal{S}(\lam_0))$ are isomorphisms and, by Theorem \ref{th:recvgl}, $h_1: \mathcal{N}_{r}(\mathcal{\mathbb{L}}(\lam_0)) \rightarrow \mathcal{N}_{r}(G(\lam_0))$ and $h_2: \mathcal{N}_{l}(\mathbb{L}(\lam_0)) \rightarrow \mathcal{N}_{l}(G(\lam_0))$ are isomorphisms,  it follows  that $k_1 = f\circ h_1: \mathcal{N}_{r}(\mathbb{L}(\lam_0)) \rightarrow \mathcal{N}_{r}(\mathcal{S}(\lam_0))$ and $k_2 = g\circ k_2 : \mathcal{N}_{l}(\mathbb{L}(\lam_0)) \rightarrow \mathcal{N}_{l}(\mathcal{S}(\lam_0))$   are isomorphisms. 

\section{Vector Spaces of Linearizations for Transfer functions}
%Consider the {\em Multivariable linear time invariant (LTI) state-space system} $\Sigma_1$ given by
%
%\begin{equation}
%\begin{aligned}
%A\left(\frac{d}{dt}\right) x(t) &= B u(t), \\
%y(t) &= C x(t) + D\left(\frac{d}{dt}\right) u(t),
%\end{aligned} \label{hssf02}
%\end{equation}
%
%where $A(\lam) = \sum_{j=0}^{m}\lam^{j}A_j \in \C[\lam]^{n \times n}$ is regular matrix polynomial of degree $m$ and $ D(\lambda) = \sum_{j=0}^{k}\lam^{j}D_j \in \C[\lam]^{r \times r}$ is a  matrix polynomial of degree $m$,  $ C \in \C^{r \times n}, B \in \C^{n \times r}$. 
%The associated system matrix is given by
%
%\begin{equation}
%\mathcal{S}(\lam) = \left[
%                       \begin{array}{c|c}
%                         A(\lam) & -B \\
%                         \hline
%                         C & D (\lambda) \\
%                       \end{array}
%                     \right] \in \C^{(n+r)\times (n+r)} \label{smh}
%\end{equation}
%
%and the transfer function of $\Sigma_1$ is given by
%
%\begin{equation}
%G(\lam) = C A(\lam)^{-1} B + D(\lambda) \in \C^{r \times r}. \label{tfunction02}
%\end{equation}

Consider $G(\lam)$ given in (\ref{tfunction01}). Then recall the first companion linearizations of $G(\lam)$ is given by 

{\scriptsize
\begin{eqnarray}
C_{1}(\lam) &= & \lam  \left[
                         \begin{array}{cccc|cccc}
                           A_{m} &  &  &  & & & & \\
                            & I_{n} &  &  & & & & \\
                            &  & \ddots &  & & & & \\
                            &  &  & I_{n} &  & & &\\
                            \hline
                            &  &  &  & D_k & & & \\
                            &  &  &  &  & I_r& & \\
                            &  &  &  &  & & \ddots & \\
                            &  &  &  &  & & & I_r \\
                         \end{array}
                       \right] \nonumber\\
                       && \quad+ \left[
                \begin{array}{cccc|cccc}
                  A_{m-1} & A_{m-2}& \cdots & A_{0}& 0& \cdots & 0& -B  \\
                   -I_{n} & 0 & \cdots & 0 & 0 & \cdots & & 0\\
                   & \ddots &  & \vdots  & \vdots & & & \vdots\\
                   &  & -I_{n} & 0 & 0 &  \cdots & & 0 \\
                   \hline
                 0 & \cdots & 0 & C &  D_{k-1} & D_{k-2} &\cdots & D_0 \\
                  & 0 & \cdots & 0 &  -I_r & 0 & \cdots & 0 \\
                  &  & \ddots & \vdots &   & \ddots & & \vdots\\
                  &  &  & 0 &   & & -I_r & 0 \\
                \end{array}
              \right]
\end{eqnarray}}

Now, we have $C_{1}(\lam)\left[\begin{array}{c}
                         \Lambda_{m-1} \otimes A(\lambda)^{-1}Bx \\
                         \Lambda_{k-1} \otimes x
                         \end{array}\right]  
                = \left[\begin{array}{cccccc}
                0 & \cdots & 0 & (G(\lambda)x)^{T} & 0 & 0 \end{array}\right]^{T}$ 
for all $x \in \C^r.$ It is equivalent to  the identity 

$$C_{1}(\lam)\left[\begin{array}{c}
                    \Lambda_{m-1} \otimes A(\lambda)^{-1}B \\
                    \Lambda_{k-1} \otimes I_r
                    \end{array}\right] 
            = \left[\begin{array}{c}
                         0 \\
                         \vdots \\
                         0 \\\hline
                        G(\lambda)  \\
                         0 \\
                         \vdots \\
                         0 \\
              \end{array}\right] 
            = \left[\begin{array}{c}
                         0 \\\hline
                         e_{1} \otimes G(\lambda) \\
              \end{array}\right].$$
If we generalize this to any arbitrary pencil then we consider the set of pencils $\mathbb{L}(\lambda) := \lambda \mathbb{X} + \mathbb{Y}$ satisfying the property

$$\mathbb{L}(\lambda)\left[\begin{array}{c}
                     \Lambda_{m-1} \otimes  A(\lambda)^{-1}B\\
                     \Lambda_{k-1} \otimes I_r
                    \end{array}\right] 
    = \mathbb{L}(\lambda)\left[\begin{array}{c} 
                         \lambda^{m-1}  A(\lambda)^{-1}B \\
                         \vdots \\
                          A(\lambda)^{-1}B \\
                          \lambda^{k-1} I_r \\
                          \vdots \\
                          I_r \\
                         \end{array}\right]
    = \left[\begin{array}{c}
                         0\\
                         \vdots \\
                         0\\
                         \hline 
                        w_1 G(\lambda) \\
                         \vdots \\
                         w_k G(\lambda)\\
                         \end{array}\right] 
    = \left[\begin{array}{c}
                         0\\
                         \hline
                          w \otimes G(\lambda)
      \end{array}\right] $$
for some vector $\left[\begin{array}{c} 0 \\\hline w \end{array}
\right]\in \C^{m+k} $ such that $w \in \C^k $. This set of pencils will be denoted by $\mathbb{L}_{1}(G)$ as a reminder that it generalizes the first companion form of $G$.

Define  

$$\Gamma_{G} = \left\{\left[\begin{array}{c}
                         0\\ \hline
                          w \otimes G(\lambda)
               \end{array}\right]  
               : w \in \mathbb{C}^{k}\right\}.$$

\begin{lemma}
$\Gamma_{G}$ is a vector space under the Kronecker product.
\end{lemma}

\begin{proof}
Let  
{\footnotesize $\left[\begin{array}{c}
                  0\\
                 v_{1}\otimes G(\lambda)\\
 \end{array}\right]$} and 
{\footnotesize $\left[\begin{array}{c}
                 0 \\
                 v_{2}\otimes G(\lambda) \\
 \end{array}\right]  \in \Gamma_{G}$ for $v_{1}, v_{2} \in \mathbb{C}^{k}$}.
Now \\
{\footnotesize $\left[\begin{array}{c}
                    0 \\
                    v_{1}\otimes G(\lambda)\\
                \end{array}\right] 
             +  \left[\begin{array}{c}
                    0 \\
                    v_{2}\otimes G(\lambda) \\
                \end{array}\right] 
             = \left[\begin{array}{c}
                    0 \\
                    v_{1}\otimes G(\lambda)+ v_{2}\otimes G(\lambda) \\
                \end{array}\right]
             = \left[\begin{array}{c}
                    0 \\
                    (v_{1}+ v_{2})\otimes G(\lambda) \\
                \end{array}\right]$},
by the properties of kronecker product and $v_{1}+v_{2}\in \mathbb{C}^{k}$. So
{\footnotesize $\left[\begin{array}{c}
                     0 \\
                    (v_{1}+ v_{2})\otimes G(\lambda) \\
                \end{array}\right] \in \Gamma_{G}$}. 
Similarly, for any scalar $\alpha$ and 
{\footnotesize $\left[\begin{array}{c}
                                 0 \\
                                 w\otimes G(\lambda) \\
                      \end{array}\right]\in \Gamma_{G}$}, 
we have {\footnotesize $\alpha\left[\begin{array}{c}
                                 0 \\
                                 w\otimes G(\lambda) \\
                      \end{array}\right]\in \Gamma_{G}$}.
All other properties of vector space are satisfying from the properties of  Kronecker product. So $\Gamma_{G}$ is a vector space.
    
\end{proof}

Notice that from the properties of Kronecker product it is easy to see that $\Gamma_{G}$ is a vector space isomorphic to $\C^{m+k}$. Now, we have the following definition.

\begin{definition}
Consider the classes of pencils
 $\mathbb{L}_{1}(G) 
     := \biggl\{\mathbb{L}(\lambda) 
      = \lambda \mathbb{X} + \mathbb{Y} 
     :=\lambda \left[\begin{array}{c|c}
                 X_{11} & 0 \\\hline
                 0 & X_{22} \\
               \end{array}\right]
        + \left[\begin{array}{c|c}
                 Y_{11} & Y_{12} \\ \hline
                 Y_{21} & Y_{22} \\
          \end{array}\right] 
    : \mathbb{X}, \mathbb{Y} \in \mathbb{C}^{(mn +kr)\times (mn +kr)}, 
     \mathbb{L}(\lambda)
     \left[\begin{array}{c}
            \Lambda_{m-1} \otimes A(\lambda)^{-1}B \\
            \Lambda_{k-1} \otimes I_r
     \end{array}\right]  \in \Gamma_{G}\biggr\}. $   
\end{definition}

%That is,  $\mathbb{L}_{1}(G) := \{\mathbb{L}(\lambda) = \lambda \mathbb{X} + \mathbb{Y} :=\left[
%                                                        \begin{array}{c|c}
%                                                        L(\lambda) & \mathcal{B} \\
%                                                         \hline
%                                                       \mathcal{C} & K(\lambda) \\
%                                                          \end{array}
%                                                        \right] : (L(\lambda), v) \in %\mathcal{L}_{1}(P),  \,\,  (K(\lambda), w) \in \mathcal{L}_{1}(P)\}.$

\begin{proposition}
For any rational matrix $G(\lambda)$,  $\mathbb{L}_{1}(G)$ is a vector space over $\mathbb{C}$. 
\end{proposition}

\begin{remark}
 Since $C_1(\lambda)$ is always in $\mathbb{L}_{1}(G)$, we see that $\mathbb{L}_{1}(G)$ is a nontrivial vector space for any rational matrix $G$.   
\end{remark}

\begin{lemma}\label{eqdL1g}
 Let $G(\lam) = C A(\lam)^{-1} B + D(\lambda) $  be any rational matrix and $\mathbb{L}(\lambda) = \lambda \mathbb{X} + \mathbb{Y}  = \lambda \left[
                                                  \begin{array}{c|c}
                                                    X_{11} & 0 \\
                                                    \hline
                                                    0 & X_{22} \\
                                                  \end{array}
                                                \right] + \left[
                                                        \begin{array}{c|c}
                                                         Y_{11} & -v e_{k}^{T} \otimes B \\
                                                         \hline
                                                        we_{m}^{T}\otimes C & Y_{22} \\
                                                          \end{array}
                                                        \right]\mathbb{C}^{mn+kr \times mn+kr}$  be any pencil. 
Then for $v \in \mathbb{C}^{m}$, $w \in \mathbb{C}^{k}$, we have 
            $$(\lambda \mathbb{X} +\mathbb{Y}).\left[
                         \begin{array}{c}
                         \left(\Lambda_{m-1} \otimes  A(\lambda)^{-1}B\right)\\ \hline
                         \Lambda_{k-1} \otimes I_r
                         \end{array}
              \right]=\left[
                         \begin{array}{c}
                         0\\
                         \hline
                          w \otimes G(\lambda)
                         \end{array}
              \right] \iff $$ $$  X \dashv\vdash Y =\left[
                                                  \begin{array}{c|c}
                                                   X_{11} \boxplus Y_{11} = v\otimes [A_{m}, A_{m-1}, \ldots , A_{0}] & -v e_{k+1}^{T} \otimes B \\
                                                    \hline
                                                    we_{m+1}^{T}\otimes C & X_{22} \boxplus Y_{22} = w\otimes [D_{k}, D_{k-1}, \ldots , D_{0}] \\
                                                  \end{array}
                                                \right]                                                            
 $$
\end{lemma}

Now, we alternatively characterized the space $\mathbb{L}_{1}(G)$  as follows:

\begin{definition}
Consider  the classes of pencils
{\scriptsize
\begin{align*}
& \mathbb{L}_{1}(G) 
     := \biggl\{\mathbb{L}(\lambda) 
      = \lambda \mathbb{X} + \mathbb{Y} 
      = \lambda \left[\begin{array}{c|c}
                X_{11} & 0 \\ \hline
                0 & X_{22} \\
            \end{array}\right] 
     + \left[\begin{array}{c|c}
            Y_{11} & -v e_{k}^{T} \otimes B \\\hline
            we_{m}^{T}\otimes C & Y_{22} \\
        \end{array}\right] :  \\ 
&\hspace*{2cm}
    X_{11} \boxplus Y_{11} 
 = v\otimes \left[\begin{array}{cccc} A_{m}  & \cdots & A_{0} \end{array}\right] 
 \quad \mbox{and} \quad 
 X_{22} \boxplus Y_{22} = w\otimes \left[\begin{array}{cccc} D_{k}  & \cdots & D_{0} \end{array}\right] \biggr \}.
\end{align*}}
Equivalently,  
\begin{align*}
&\mathbb{L}_{1}(G) 
      := \biggl\{\left[\begin{array}{c|c}
                L(\lambda) & -v e_{k}^{T} \otimes B  \\\hline
                we_{m}^{T} \otimes C & K(\lambda) \\
        \end{array}\right] :  
(L(\lambda), v) \in \mathcal{L}_{1}(A),  \quad  (K(\lambda), w) \in \mathcal{L}_{1}(D)\biggr\}.
\end{align*}
\end{definition}
%
%Then 
%$$\left(\lambda \mathbb{X} + \mathbb{Y}\right)
%       \left[\begin{array}{c}
%         \Lambda_{m-1} \otimes  A(\lambda)^{-1}B\\ \hline
%         \Lambda_{k-1} \otimes I_r
%       \end{array}\right]
%=\left[\begin{array}{c}
%        0\\ \hline
%        w \otimes G(\lambda)
 %\end{array}\right].$$

\begin{theorem}
Let 
$\mathbb{L}(\lambda)
    =\left[\begin{array}{c|c}
         L(\lambda) & -ve_k^T\otimes B \\ \hline
         we_m^T\otimes C & K(\lambda)
    \end{array}\right]\in \mathbb{L}_{1}(G)$.
Then   
$$\mathbb{L}(\lambda) 
     \left[\begin{array}{c}
            \Lambda_{m-1} \otimes  A(\lambda)^{-1}B\\ \hline
            \Lambda_{k-1} \otimes I_r
     \end{array}\right]
 =\left[\begin{array}{c}
            0\\\hline
            w \otimes G(\lambda)
  \end{array}\right].$$
\end{theorem}

\begin{proof}
We have
\begin{align*}
       &\mathbb{L}(\lambda)
          \left[\begin{array}{c}
          \Lambda_{m-1} \otimes  A(\lambda)^{-1}B\\ \hline
          \Lambda_{k-1} \otimes I_r
    \end{array}\right]\\
=& \left[\begin{array}{c|c}
       L(\lambda) & -ve_k^T\otimes B \\ \hline
       we_m^T\otimes C & K(\lambda)
    \end{array}\right]
    \left[\begin{array}{c}
        \Lambda_{m-1} \otimes  A(\lambda)^{-1}B\\ \hline
        \Lambda_{k-1} \otimes I_r
    \end{array}\right]\\
=& \left[\begin{array}{c}
        L(\lambda) (\Lambda_{m-1}\otimes I_n)A(\lambda)^{-1}B-ve_k^T\Lambda_{k-1}\otimes B \\ \hline
        (we_m^T\Lambda_{m-1}\otimes C)A(\lambda)^{-1}B+K(\lambda)(\Lambda_{k-1}\otimes I_r
    \end{array}\right]\\
=& \left[\begin{array}{c}
        (v\otimes A(\lambda))A(\lambda)^{-1}B-v\otimes B\\ \hline
        (w\otimes C)A(\lambda)^{-1}B+w\otimes D(\lambda)\end{array}\right]\\
=& \left[\begin{array}{c} 
        v\otimes B-v\otimes B \\ \hline
        w\otimes (CA(\lambda)^{-1}B+D(\lambda))
    \end{array}\right]
= \left[\begin{array}{c} 
         0 \\ \hline
         w\otimes G(\lambda)
    \end{array}\right]
\end{align*} 
\end{proof}

\begin{exam}
Consider $G(\lambda) = C(A_0 - \lambda A_1)^{-1} B + D_2 \lambda^{2} + D_1 \lambda + D_0 $ and 
let 
$$X = \left[\begin{array}{c|cc}
               -A_1 & 0 & 0 \\\hline
                0  & D_{2} & D_{1}+D_{0}  \\
                0  &  D_{2} & 2D_{1}+D_{2}  \\
        \end{array}\right] \quad \mbox{and}\quad 
Y = \left[\begin{array}{c|cc}
          A_0 & 0 & -B  \\\hline
          C & - D_{0}& D_{0}  \\
          C & -D_{2}-D_{1} & D_{0}  \\
    \end{array}\right].$$
Then 
$$(\lambda X +Y)\left[\begin{array}{c}
                        (A_0-\lambda A_1)^{-1}B \\
                         \lambda I \\
                         I \\
                       \end{array}
                     \right]
= \left[\begin{array}{c}
            0 \\\hline
            G(\lambda) \\
            G(\lambda) \\
  \end{array}\right] 
= \left[\begin{array}{c}
            0 \\\hline
            \left[\begin{array}{c}
                1 \\
                1 \\
            \end{array}\right]\otimes
            G(\lambda)
            \end{array}\right]. $$                                 
\end{exam}

\begin{exam}
Consider $G(\lambda) = C(A_0 + \lambda A_1 +\lambda^2 A_2)^{-1} B + D_2 \lambda^{2} + D_1 \lambda + D_0 $ and 
let 
{\scriptsize $$X =  \left[\begin{array}{cc|cc}
          A_2 & A_1 +A_0 & 0 &  0 \\
          A_2 & 2 A_1 +A_2 &  0 &  0 \\\hline
          0 & 0  & D_{2} & D_{1}+D_{0}  \\
          0 & 0  &  D_{2} & 2D_{1}+D_{2}  \\
      \end{array}\right] 
\quad \mbox{and}\quad 
Y =\left[\begin{array}{cc|cc}
        -A_0 & A_0  & 0 & -B  \\
         -A_2- A_1 & A_0  & 0 & -B  \\\hline
         0 & C & - D_{0}& D_{0}  \\
         0 &  C & -D_{2}-D_{1} & D_{0}  \\
    \end{array}\right].$$ }
Then 
$$(\lambda X +Y)\left[\begin{array}{c}
                \lambda (A_0+\lambda A_1+ \lambda^2 A_2)^{-1}B \\
                (A_0+\lambda A_1+ \lambda^2 A_2)^{-1}B \\
                \lambda I \\
                I \\
            \end{array}\right] 
= \left[\begin{array}{c}
            0 \\
            0 \\\hline
            G(\lambda) \\
            G(\lambda) \\
  \end{array}\right] 
= \left[\begin{array}{c}
            0 \\
            0 \\\hline
            \left[\begin{array}{c}
            1 \\
            1 \\
            \end{array}\right]\otimes G(\lambda)
  \end{array}\right]$$
%{\bf Check the example}
\end{exam}

\begin{theorem}
[Characterization of pencils in $\mathbb{L}_{1}(G)$] \label{chtL1G}
Let $G(\lambda):=D(\lambda)+C(A(\lambda))^{-1}B$ be an $r\times r$
  rational matrix, and 
  $\left[\begin{array}{c}
        v \\
        w \\
    \end{array}\right]\in\, \mathbb{C}^{m+k}$
any vector. Then the set of pencils in $\mathbb{L}_{1}(G)$ with right ansatz vector 
 $\left[\begin{array}{c}
       v \\
       w \\
  \end{array}\right]$ 
consists of all $\mathbb{L}(\lambda)= \lambda  \mathbb{X}+\mathbb{Y}$ such that
\[\mathbb{X} = \bordermatrix{ & n  & (m-1)n & r  & (k-1)r \cr
                   &  v\otimes A_{m} & W   & 0 & 0\cr
                   & 0 & 0 & w\otimes D_{k} & W_{1} \cr }\]
and 
\[\mathbb{Y} = \bordermatrix{ &  (m-1)n & n  &  (k-1)r & r\cr
          &  v \otimes [A_{m-1} \cdots A_{1} ] -W & v \otimes A_{0}  & 0 & 
          -ve_{k}^{T}\otimes B  \cr
        & 0 & we_{m}^{T}\otimes C  & -W_{1}+ w \otimes [D_{k-1} \cdots D_{1}] &  w \otimes D_{0}\cr}\]
with $W\in\, \mathbb{F}^{(nm)\times n(m-1)}$, $W_{1}\in \mathbb{F}^{(rk)\times r(k-1)}$ chosen arbitrarily.
\end{theorem}

\begin{proof}
Consider the map
$\mathcal{F}:\mathbb{L}_{1}(G) \longrightarrow \Gamma_{G}$ defined by

$$\mathcal{F}(L(\lambda)) 
   = L(\lambda)\left[\begin{array}{c}
        \Lambda_{m-1} \otimes A(\lambda)^{-1}B \\
        \Lambda_{k-1} \otimes I_r
    \end{array}\right].$$
Claim : $\mathcal{F}$ is linear and surjective. \\
Now,

\begin{align*}
\mathcal{F}\left( \left(\alpha L_{1}+L_{2}\right)(\lambda)\right) 
& = \left(\alpha L_{1}+L_{2}\right)(\lambda)
    \left[\begin{array}{c}
       \Lambda_{m-1} \otimes A(\lambda)^{-1}B \\
       \Lambda_{k-1} \otimes I_r
    \end{array}\right] \\
&= \left(\alpha L_{1}(\lambda)+L_{2}(\lambda)\right)
   \left[\begin{array}{c}
        \Lambda_{m-1} \otimes A(\lambda)^{-1}B \\
        \Lambda_{k-1} \otimes I_r
    \end{array}\right]\\
& = \alpha L_{1}(\lambda) 
     \left[\begin{array}{c}
        \Lambda_{m-1} \otimes A(\lambda)^{-1}B \\
        \Lambda_{k-1} \otimes I_r
      \end{array}\right]
  + L_{2}(\lambda)
   \left[\begin{array}{c}
        \Lambda_{m-1} \otimes A(\lambda)^{-1}B \\
        \Lambda_{k-1} \otimes I_r
    \end{array}\right] \\
&= \alpha \mathcal{F}(L_{1}(\lambda)) +  \mathcal{F}(L_{2}(\lambda))
\end{align*}
So, $\mathcal{F}$ is linear.
Let $\left[\begin{array}{c}
         0 \\
     w\otimes G(\lambda) \\
       \end{array}\right]$
be an arbitrary element of $\Gamma_{G}$. 
For $\left[\begin{array}{c}
    0 \\
  w \\
     \end{array}\right] \in \mathbb{C}^{m+k}$, construct 
 
$$ X_{\left[
     \begin{array}{c}
      v \\
      w \\
     \end{array}\right] }
 = \left[\begin{array}{cc|cc}  
        v \otimes A_{m} & 0  & 0 & 0 \\\hline
        0 & 0 &  w \otimes D_{k} & 0\\
    \end{array}\right] $$
and 
$$ Y_{\left[\begin{array}{c}
      v \\
      w \\
  \end{array}\right]} 
= \left[\begin{array}{cc|cc}
    v \otimes [A_{m-1} \cdots A_{1}]  &  v \otimes A_0 &  0 & -ve_{k}^{T} \otimes B  \\\hline
    0 & w e_{m}^{T} \otimes C & w \otimes [D_{k-1} \cdots D_{1}] & w \otimes D_0 \\
  \end{array}\right].$$
 %\textcolor{red}{Sizes of $X_{\left[\begin{array}{c} v \\ w\end{array}\right]}$ and $Y_{\left[\begin{array}{c} v \\ w\end{array}\right]}$ do not match. (now check it)}
 Then
 \begin{align*}
& X_{\left[\begin{array}{c}
        v \\
        w \\
      \end{array}\right]}\boxplus 
    Y_{\left[\begin{array}{c}
        v \\
        w \\
      \end{array}\right]}  
 &= \left[\begin{array}{c|c}
       v \otimes [A_{m} \,\, A_{m-1} \cdots A_{0}] &  v \otimes [0 \,\,\cdots 0 \,\, -e_{k}^{T} \otimes B] \\\hline
       w \otimes [0 \,\, \cdots 0 \,\,\, e_{m}^{T} \otimes C]  & w \otimes [D_{k} \,\, D_{k-1} \cdots D_{0}]  \\
    \end{array}\right] ,                                  
 \end{align*} see proof of Theorem~\ref{cl1s}.
So, by the Lemma~\ref{eqdL1g}  the pencil 
$\mathbb{L}_{\left[\begin{array}{c}
                        v \\
                        w \\
                \end{array}\right]}(\lambda) 
:= \lambda X_{\left[\begin{array}{c}
                v \\
                w \\
              \end{array}\right]}
+ Y_{\left[\begin{array}{c}
               v \\
               w \\
          \end{array}\right]}$
is an $\mathcal{F}$-preimage of 
$\left[\begin{array}{c}
         v \\
         w \\
 \end{array}\right]\otimes G(\lambda)$. So $\mathcal{F}$ is onto. The set of all $\mathcal{F}$-preimages of 
 $\left[\begin{array}{c}
         v \\
         w\\
\end{array}\right]\otimes G(\lambda)$ is then $\mathbb{L}_{\left[\begin{array}{c}
               v \\
               w \\
            \end{array}\right]}(\lambda)+\ker\mathcal{F}$.
So, now, we have to calculate $\ker\mathcal{F}$.

$$\ker\mathcal{F} = \left\{L(\lambda)\in \mathbb{L}_{1}(G) : \mathcal{F}(L(\lambda)) = 0\right\} =\{L(\lambda)\in \mathbb{L}_{1}(G) : \mathbb{X} \dashv\vdash \mathbb{Y} = 0\}.$$ 
%
%i.e., the kernel of $\mathcal{F}$ consists of all pencils $\lambda X+ Y$ such that $X\boxplus Y = 0$. Let
%$$ X = \left[
%       \begin{array}{cc}
 %        X_{11} & X_{12} \\
%         X_{21} & X_{22} \\
%       \end{array}
%     \right] \mbox{    and    } Y = \left[
%       \begin{array}{cc}
%         Y_{11} & Y_{12} \\
%         Y_{21} & Y_{22} \\
%       \end{array}
 %    \right]
 %$$
Next, we find the kernel of $\mathcal{F}$. Let 
$$\mathbb{L}(\lambda)=\lambda\mathbb{X}+\mathbb{Y}
                     =\left[\begin{array}{c|c}
                      \lambda X +Y & -ve_{k}^{T}\otimes B \\ \hline
                       we_{m}^{T}\otimes C & \lambda P + Q \end{array}\right] \in \ker \mathcal{F}.
                       $$
Then $\mathbb{X} \dashv\vdash \mathbb{Y} =0$ implies
$$\left[\begin{array}{c|c}
        X  & 0 \\ \hline
        0 &  P  
  \end{array}\right] \boxplus
  \left[\begin{array}{c|c}
        Y & -ve_{k}^{T}\otimes B \\ \hline
        we_{m}^{T}\otimes C &  Q 
  \end{array}\right]=0$$

\begin{align*}
& \implies 
  \left[\begin{array}{c|c}
     X\boxplus Y  & 0 \boxplus -ve_{k}^{T}\otimes B \\\hline
     0 \boxplus we^{T}_{m}\otimes C & P \boxplus Q
  \end{array}\right]=0\\   
& \implies  X\boxplus Y  = 0,\quad  P \boxplus Q= 0,\quad  -v e_{k}^{T}\otimes B=0,\quad\mbox{and}\quad  we_{m}^{T}\otimes C=0 \\
& \implies  
  X = \left[\begin{array}{cc}
         0 & - W_1 \\
       \end{array}\right] \quad \mbox{and}\quad 
  Y =  \left[\begin{array}{cc}
         W_1 & 0 \\
       \end{array}\right]  \\
& \quad\qquad  
  P =  \left[\begin{array}{cc}
        0 & - W_2 \\
       \end{array}\right] \quad\mbox{and}\quad 
  Q =  \left[\begin{array}{cc}
        W_2 & 0 \\
       \end{array}\right] 
\end{align*}

Thus the definition of shifted sum then implies that $X$ and $Y$ must have the form
$$ X = \left[\begin{array}{cc|cc}
         0 & W & 0 & 0\\\hline
        0 & 0 & 0 & W_{1} \\
       \end{array}\right] \quad\mbox{and}\quad 
   Y = \left[\begin{array}{cc|cc}
        -W & 0  & 0 & 0\\\hline
         0 & 0 &  -W_{1} & 0 \\
       \end{array}\right].$$
Thus 
{\scriptsize \begin{align*}
\mathbb{L}_{\left[\begin{array}{c}
               v \\
               w \\
    \end{array}\right]}(\lambda)+&\ker (\mathcal{F}) \\= 
&\lambda
\bordermatrix{ & n  & (m-1)n \cr
                   &  v\otimes A_{m} & W  & 0 & 0 \cr
                   & 0 & 0 & w\otimes D_{k} & W_{1} \cr}\\
&+\bordermatrix{ &  (m-1)n & n \cr
          &  v \otimes [A_{m-1} \cdots A_{1} ] -W  &  v \otimes A_{0} & 0 & -ve_{k}^{T}\otimes B \cr
         & 0 & we_{m}^{T}\otimes C  & w \otimes [D_{k-1} \cdots D_{1}] -W_{1} & w \otimes D_0 \cr
                   }
\end{align*}}

\end{proof}

\begin{corollary}
We have $\dim \mathbb{L}_{1}(G) =  \dim L_{1}(A) + \dim L_{1}(D).$  
\end{corollary}

%\begin{proof}
%We have 
%\begin{align*}
%\dim \mathbb{L}_{1}(G) & = \dim \ker \mathcal{F} + \dim \Gamma (G) \\
%& = m+k + (m-1)n . mn + (k-1)r. rk \\
%   & = m+ m (m-1)n^2 + k+ k (k-1)r^2 \\
 %  & =  \dim \mathcal{L}_{1}(A) + \dim \mathcal{L}_{1}(D).     
%\end{align*}
%\end{proof}

\begin{corollary} \label{CorL1}
If $\mathbb{L}(\lambda) = \lambda X +Y \in \mathbb{L}_{1}(G) $ with ansatz  vector 
$ \left[
       \begin{array}{c}
        v  \\
        w  \\
       \end{array}
     \right]=\left[
       \begin{array}{c}
        \alpha e_1  \\
         \beta e_1  \\
       \end{array}
     \right], $ then $$ X = \left[
       \begin{array}{cc|cc}
         \alpha A_m & X_{12} & 0 & 0 \\
         0 & -Z & 0 & 0 \\
         \hline
          0 & 0 & \beta D_k & S_{12} \\
           0 & 0 & 0 & -Z_1\\
       \end{array}
     \right], \,\,  Y = \left[
       \begin{array}{cc|cc}
        Y_{11} & \alpha A_0 & 0 & -ve_{k}^{T}\otimes B \\
         Z & 0 & 0 & 0 \\
         \hline
          0 & we_{m}^{T}\otimes C  & S_{11 } & \beta D_0 \\
           0 & 0 & Z_1 & 0 \\
       \end{array}
     \right],  $$ where $Z \in \mathbb{C}^{(m-1)n \times (m-1)n}, Z_1 \in \mathbb{C}^{(k-1)r \times (k-1)r}$. 
\end{corollary}

\begin{proof}
Proof directly follows from the previous Theorem~\ref{chtL1G}.     
\end{proof}

\begin{observation}
Observe that for $\left[
       \begin{array}{c}
        v  \\
        w  \\
       \end{array}
     \right]=\left[
       \begin{array}{c}
         e_1  \\
          e_1  \\
       \end{array}
     \right],$  $Z = -I_{(m-1)n}$ and $Z_1 = -I_{(k-1)r}$,  $C_{1}(\lambda)$ follows the pattern of Corollary~\ref{CorL1}.  
\end{observation}

\begin{theorem}\label{evrL1}
Let $G(\lambda)$ be an $r \times r$ matrix rational function, $\lambda$ is not a pole of $G(\lambda)$, and $L(\lambda)$ be any pencil in $\mathbb{L}_{1}(G)$ with non zero right ansatz vector $\left[
          \begin{array}{c}
            0 \\
            w \\
          \end{array}
        \right].
$ Then $x \in \mathbb{C}^{r}$ is an eigenvector for $G(\lambda)$ with finite eigenvalue $\lambda \in \mathbb{C}$ iff  $\left[
                  \begin{array}{c}
                    \Lambda_{m-1} \otimes A(\lambda)^{-1}B x \\
                    \Lambda_{k-1} \otimes x \\
                  \end{array}
                \right]$ is an eigenvector of $\mathbb{L}(\lambda)$ with the eigenvalue $\lambda$. If $G$ is regular, and $\mathbb{L} \in \mathbb{L}_{1}(G)$ is a linearization for $G$, then every eigenvector of $\mathbb{L}$ with
finite eigenvalue $\lambda$ is of the form
                $\left[
                  \begin{array}{c}
                    \Lambda_{m-1} \otimes A(\lambda)^{-1}B x \\
                    \Lambda_{k-1} \otimes x \\
                  \end{array}
                \right]$ for some eigenvector $x$ of $G(\lambda)$.
\end{theorem}

\begin{proof}
We have 
\begin{align*}
 \mathbb{L}(\lambda).\left[
                  \begin{array}{c}
                    \Lambda_{m-1} \otimes A(\lambda)^{-1}B x \\
                    \Lambda_{k-1} \otimes x \\
                  \end{array}
                \right] & = \mathbb{L}(\lambda) . \left[
                                       \begin{array}{c}
                                        (\Lambda_{m-1} \otimes A(\lambda)^{-1}B )(1 \otimes x)  \\
                                         ( \Lambda_{k-1} \otimes I_r) (1 \otimes x) \\
                                       \end{array}
                                     \right] \\
                & = \mathbb{L}(\lambda)\left[
                           \begin{array}{c}
                            (\Lambda_{m-1} \otimes A(\lambda)^{-1}B ) \\
                             \Lambda_{k-1} \otimes I_r \\
                           \end{array}
                         \right](1 \otimes x) \\
                                        & = \left(\left[
                                                   \begin{array}{c}
                                                     0 \\
                                                     \hline
                                                    w \otimes G(\lambda) \\
                                                   \end{array}
                                                 \right]
                         \right)(1 \otimes x)
\end{align*}
\begin{equation}
\Rightarrow \mathbb{L}(\lambda). \left[
                  \begin{array}{c}
                    \Lambda_{m-1} \otimes A(\lambda)^{-1}B x \\
                    \Lambda_{k-1} \otimes x \\
                  \end{array}
                \right] = \left[
                          \begin{array}{c}
                           0 \\
                           \hline
                             w \otimes G(\lambda)x \\
                          \end{array}
                        \right].
\end{equation}
Let $\left[
                  \begin{array}{c}
                    \Lambda_{m-1} \otimes A(\lambda)^{-1}B x \\
                    \Lambda_{k-1} \otimes x \\
                  \end{array}
                \right] $ is an eigenvector of $\mathbb{L}(\lambda)$ $\Rightarrow \mathbb{L}(\lambda). \left[
                  \begin{array}{c}
                    \Lambda_{m-1} \otimes A(\lambda)^{-1}B x \\
                    \Lambda_{k-1} \otimes x \\
                  \end{array}
                \right] = 0$ for some $w \in \C^k$. Thus, we have 
                $\left[
                 \begin{array}{c}
                           0 \\
                           \hline
                             w \otimes G(\lambda)x \\
                          \end{array}
                        \right] = 0  \Rightarrow G(\lambda)x = 0$. 
                        So, $x$ is an eigenvector of $G(\lambda)$.

Conversely,
let $x\in \mathbb{C}^{r}$ is an eigenvector of $G(\lambda)$ with finite eigenvalue $\lambda \in \mathbb{C}$.
So $G(\lambda)x = 0$ and we have  $\mathbb{L}(\lambda) \left[
                  \begin{array}{c}
                    \Lambda_{m-1} \otimes A(\lambda)^{-1}B x \\
                    \Lambda_{k-1} \otimes x \\
                  \end{array}
                \right]  = 0
$.
So, $\left[
                  \begin{array}{c}
                    \Lambda_{m-1} \otimes A(\lambda)^{-1}B x \\
                    \Lambda_{k-1} \otimes x \\
                  \end{array}
                \right] $ is an eigenvector for $\mathbb{L}(\lambda)$ corresponding an eigenvalue $\lambda$. 
Let $z := \left[
                  \begin{array}{c}
                    \Lambda_{m-1} \otimes A(\lambda)^{-1}B x \\
                    \Lambda_{k-1} \otimes x \\
                  \end{array}
                \right] $. For showing the only eigenvector of $\mathbb{L}(\lambda)$ is of the form $z$, we have to 
show that $T : \textrm{Null}(G(\lambda)) \rightarrow \textrm{Null}(\mathbb{L}(\lambda))$ defined by
$$Tx = \left[
                  \begin{array}{c}
                    \Lambda_{m-1} \otimes A(\lambda)^{-1}B x \\
                    \Lambda_{k-1} \otimes x \\
                  \end{array}
                \right] $$ is an linear isomorphism. Clearly, $T$ is linear. Let $Tx = 0$. Then we have
$$ \left[
                  \begin{array}{c}
                    \Lambda_{m-1} \otimes A(\lambda)^{-1}B x \\
                    \Lambda_{k-1} \otimes x \\
                  \end{array}
                \right]  = 0 \Rightarrow \Lambda_{k-1} \otimes x = 0 \Rightarrow x = 0.$$
So $T$ is one-one. One can also show that $T$ is onto. Hence  proved. 
\end{proof}

Now, consider $G(\lam)$ given in (\ref{tfunction01}). The second companion form $C_2(\lambda)$ of $G(\lambda)$ or $\mathcal{S}(\lambda)$ \cite{beh22} is given by

\begin{eqnarray}
  \mathcal{C}_2(\lam) &=& \lam 
         \left[\begin{array}{cccc|cccc}
                           A_{m} &  &  &  & & & & \\
                            & I_{n} &  &  & & & & \\
                            &  & \ddots &  & & & & \\
                            &  &  & I_{n} & & & & \\\hline
                            &  &  &  & D_{k} & & & \\
                            &  &  &  &  & I_r & & \\
                            &  &  &  &  & & \ddots & \\
                            &  &  &  &  & &  & I_r \\
         \end{array}\right]\nonumber\\
                       && \quad+ \left[
         \begin{array}{cccc|cccc}
            A_{m-1} & -I_{n} &  &  & 0 & & & \\
            A_{m-2} & 0 & \ddots &  & \vdots & & & \\
            \vdots & \ddots &  & -I_{n} & 0 & & \ddots & \\
            A_{0} & \cdots & 0 & 0  & -B &  0 & \cdots & 0 \\\hline
            0  &  &  &   &  D_{k} & -I_r & \cdots & 0 \\
            \vdots  &  \ddots  &  &   &  D_{k-1} & & & \vdots\\
            0  &  &   \ddots &   &  \vdots & & & -I_r\\
            C  & 0 & \cdots &  0 &  D_0 & 0 & \cdots & 0\\
         \end{array}\right].\label{C2}
\end{eqnarray} 
Then for $C_2(\lambda)$  we have 
$$ 
   \left[\begin{array}{c}
        \left(\Bar{\Lambda}_{m-1} \otimes (-C A(\lambda)^{-1})^{*}\right)y\\
        (\Bar{\Lambda}_{k-1} \otimes I_r) y
   \end{array}\right]^{*}
C_{2}(\lam) = 
    \left[\begin{array}{ccc|ccc}
     0 & \cdots & 0 & y^{*}G(\lambda) & 0 & 0
     \end{array}\right]^{T}$$
for all $y \in \C^r.$ It is equivalent to  the identity 
\begin{align*}
\left[\begin{array}{c}
      \Bar{\Lambda}_{m-1} \otimes (-C A(\lambda)^{-1})^{*} \\
      \Bar{\Lambda}_{k-1} \otimes I_r
    \end{array}\right]^{*} C_{2}(\lam)  
&= \left[\begin{array}{ccc|ccc}
    0 &\cdots &0 &G(\lambda) &\cdots & \\
  \end{array}\right] \\
&= \left[\begin{array}{c|c}
    0^{T}   & e_{1}^{T} \otimes G(\lambda) \\
  \end{array}\right].
\end{align*}  
If we generalize this to any arbitrary pencil then we consider the set of pencils $\mathbb{L}(\lambda) := \lambda \mathbb{X} + \mathbb{Y}$ satisfying the property
\begin{align*}
&\left[\begin{array}{c}
    \Bar{\Lambda}_{m-1} \otimes (-C A(\lambda)^{-1})^{*}\\
    \Bar{\Lambda}_{k-1} \otimes I_r
\end{array}\right]^{*} \mathbb{L}(\lambda) \\
&= \left[\left(\begin{array}{cc}
      \left(\Bar{\Lambda}_{m-1} \otimes  (-CA(\lambda)^{-1})^*\right)^{*} &
     (\Bar{\Lambda}_{k-1} \otimes I_r)^{*}
  \end{array}\right)\right]\mathbb{L}(\lambda)\\ 
&=  \left[\begin{array}{ccc|ccc} 
     -\lambda^{m-1}  CA(\lambda)^{-1} &
     \cdots &  -CA(\lambda)^{-1} &
     \lambda^{k-1} I_r &\cdots &I_r \\
    \end{array}\right]\mathbb{L}(\lambda) \\
&= \left[\begin{array}{ccc|ccc}
     0 &\cdots &0 &w_1 G(\lambda) & \cdots &w_k G(\lambda)\\
    \end{array}\right]\\
&= \left[\begin{array}{c|c}
     0^{T} & w^{T} \otimes G(\lambda)
    \end{array}\right] 
\end{align*}
for some vector $\left[\begin{array}{c}
                         0\\\hline
                          w 
                \end{array}\right]\in \C^{m+k} $ 
such that $w \in \C^k $. This set of pencils will be denoted by $\mathbb{L}_{2}(G)$ as a reminder that it generalizes the second companion form of $G$.

Define $\mathbb{W}_{G} :=\left \{ \left[
                         \begin{array}{c|c}
                         0 &
                          w^{T} \otimes G(\lambda)
                         \end{array}
              \right]:  \left[
                         \begin{array}{c}
                         0\\
                         \hline
                          w 
                         \end{array}
              \right]\in \C^{m+k} \right\}$

\begin{definition}
 $\mathbb{L}_{2}(G) := \biggl\{\mathbb{L}(\lambda) = \lambda \mathbb{X} + \mathbb{Y} 
 = \lambda \left[\begin{array}{c|c}
                X_{11} & 0 \\\hline
                0 & X_{22} \\
            \end{array}\right] 
+ \left[\begin{array}{c|c}
        Y_{11} & Y_{12} \\\hline
        Y_{21} & Y_{22} \\
  \end{array}\right]  
: \mathbb{X}, \mathbb{Y} \in \mathbb{C}^{(mn +kr)\times (mn +kr)}, \quad \left[\begin{array}{c}
        \Bar{\Lambda}_{m-1} \otimes (-C A(\lambda)^{-1})^{*} \\
        \Bar{\Lambda}_{k-1} \otimes I_r \\
      \end{array}\right]^{*}
\mathbb{L}(\lambda) \in \mathbb{W}_{G}\biggr\}. $   
\end{definition}

\begin{definition}
Consider the classes of pencils

\begin{align*}
\mathbb{L}_{2}(G) := \biggl\{\mathbb{L}(\lambda) &= \lambda \mathbb{X}+\mathbb{Y} = \lambda 
    \left[\begin{array}{c|c}
          X_{11} & 0 \\\hline
          0 & X_{22} \\
    \end{array}\right] 
 +  \left[\begin{array}{c|c}
         Y_{11} & - e_{m}z^{T}\otimes B \\\hline
         e_{k}s^{T}\otimes C & Y_{22} \\
    \end{array}\right] : \\
 &X_{11} \widehat{\boxplus} Y_{11} 
 =  s^{T}\otimes\left[\begin{array}{c}
                     A_{m} \\
                     \vdots \\
                     A_{0} \\
                \end{array}\right],
 \qquad
 X_{22} \widehat{\boxplus} Y_{22} 
 = z^{T}\otimes \left[\begin{array}{c}
                    D_{k} \\
                    \vdots \\
                    D_{0} \\
                \end{array}\right]
\biggr\}.
\end{align*}
Equivalently, 

{\small
$$\mathbb{L}_{2}(G) 
:= \left\{\mathbb{L}(\lambda) 
 =   \left[\begin{array}{c|c}
        L(\lambda) & - e_{m}z^{T}\otimes B \\\hline
        e_{k}s^{T}\otimes C & K(\lambda) \\
     \end{array}\right] 
: (L(\lambda), s) \in \mathcal{L}_{2}(A), \quad (K(\lambda), z) \in \mathcal{L}_{2} (D)\right\}. $$}
Then 
$$\left[\begin{array}{cc}
          \Lambda_{m-1}^{T}\otimes -C A(\lambda)^{-1} &  \Lambda_{k-1}^{T} \otimes I_r  \\
       \end{array}\right]\mathbb{L}(\lambda) 
    = \left[\begin{array}{c|c}
            0 &  z^{T}\otimes G(\lambda)  \\
      \end{array}\right].
$$   
\end{definition}

\begin{exam}
Consider $G(\lambda) = C(A_0 - \lambda A_1)^{-1} B + D_2 \lambda^{2} + D_1 \lambda + D_0 $ and 
let $$X = \left[\begin{array}{c|cc}
                             -A_1 & 0 & 0 \\\hline
                         0  & D_{2} & 0  \\
                           0  &  0 & -D_0  \\
                        \end{array}\right] 
                         \quad\mbox{and}\quad 
                         Y = \left[\begin{array}{c|cc}
                                 A_0 & - B & 0  \\\hline
                                0 & D_{1}& D_{0}  \\
                                C & D_{0} & 0  \\
                             \end{array}\right].
$$
Then $$\left[
         \begin{array}{c|cc}
         -C(A_0-\lambda A_1)^{-1} & \lambda I_r & I_r  \\
         \end{array}
       \right](\lambda X+Y) = \left[
                                \begin{array}{c|cc}
                                 0 &  G(\lambda) & 0 \\
                                \end{array}
                              \right] = \left[
                                          \begin{array}{c|c}
                                            0 & e_{1}^{T}\otimes G(\lambda)  \\
                                          \end{array}
                                        \right]
$$
\end{exam}

Consider $G(\lambda) = D(\lambda) +CA(\lambda)^{-1} B.$ Then the transpose of  $G(\lambda)$ is given by $G^{T}(\lambda) = D^{T}(\lambda) +B^{T}(A(\lambda)^{-1})^{T} C^{T},$ where $D^{T}(\lambda) =  \sum_{i=0}^{k} D_{i}^{T} \lambda^{i}$ and $A^{T}(\lambda) =  \sum_{i=0}^{k} A_{i}^{T} \lambda^{i}$

\begin{lemma}\label{L1L2r}
 We have $\mathbb{L}_{2} (G) = (\mathbb{L}_{1}(G^{T}))^{T}.$  That is, $ \mathbb{L}(\lambda) \in \mathbb{L}_{1}(G^{T}) \iff \mathbb{L}^{T} \in\mathbb{L}_{2}(G).$  
\end{lemma}

{\small
\begin{proof}
Note that 
\begin{align*}
  \mathbb{L}(\lambda) \in \mathbb{L}_{1}(G^{T})
&  \iff \mathbb{L}(\lambda)
   \left[\begin{array}{c}
   \left(\Lambda_{m-1} \otimes (A(\lambda)^{-1})^{T}C^{T}\right) \\
   \Lambda_{k-1} \otimes I_r
   \end{array}\right] = 
   \left[\begin{array}{c}
   0 \\ \hline
   w \otimes G(\lambda)^{T} \\
   \end{array}\right] \\
&  \iff
   \left[\begin{array}{c}
      \Lambda_{m-1} \otimes (A(\lambda)^{-1})^{T}C^{T} \\
      \Lambda_{k-1} \otimes I_r
      \end{array}\right]^{T}  
     \mathbb{L}(\lambda)^{T}= 
    \left[\begin{array}{c}
      0 \\ \hline
      w \otimes G(\lambda)^{T} \\
    \end{array}\right]^{T}    \\
& \iff 
     \left[\begin{array}{cc}
         \Lambda_{m-1}^{T} \otimes CA(\lambda)^{-1} &  \Lambda_{k-1}^{T} \otimes I_r\\
      \end{array}\right] \mathbb{L}(\lambda)^{T} 
  = \left[\begin{array}{cc}
         0^{T} &  w^{T} \otimes G(\lambda)  \\
     \end{array}\right] \\
& \iff \mathbb{L}^{T} \in\mathbb{L}_{2}(G). 
\end{align*}    
\end{proof}}

\begin{definition}
A left eigenvector of an $n \times n$ matrix
polynomial $G$ associated with a finite eigenvalue $\lambda$ is a nonzero vector $y \in \mathbb{C}^{r}$ such
that $y^{*}G(\lambda) = 0.$    
\end{definition}

Now, we have the following theorem. 

\begin{theorem}
Let $G(\lambda)$ be an $n \times n$ matrix rational function, $\lambda$ is not a pole of $G(\lambda)$, and $\mathbb{L}(\lambda)$ be any pencil in $\mathbb{L}_{2}(G)$ with non zero left ansatz vector $\left[\begin{array}{cc}0 &  z^{T}\end{array}\right]$. Then $y \in \mathbb{C}^{r}$ is a left eigenvector for $G(\lambda)$ with finite eigenvalue $\lambda \in \mathbb{C}$ iff
$\left[\begin{array}{cc}
     \bar{\Lambda}_{m-1}\otimes (-C A(\lambda)^{-1})^{*} y   & \\
     \bar{\Lambda}_{k-1}\otimes y \\
   \end{array}\right]$
is a left eigenvector for $\mathbb{L}(\lambda)$ with eigenvalue $\lambda$. If in addition, $G$ is regular and $\mathbb{L} \in \mathbb{L}_{2}(G)$ is a linearization for $G$, then every left eigenvector of $\mathbb{L}$ with finite eigenvalue $\lambda$ is of the form
$\left[
   \begin{array}{cc}
     \bar{\Lambda}_{m-1} \otimes (-C A(\lambda)^{-1})^{*} y   & \\
     \bar{\Lambda}_{k-1}\otimes y \\
   \end{array}
 \right]$
 for some left eigenvector $y$ of $G$.
\end{theorem}

\begin{proof}
     Proof is similar as theorem~\ref{evrL1}. 
\end{proof}

\section{Double ansatz pencils for $G$}

\begin{definition}
$\mathbb{DL}(G) := \mathbb{L}_{1}(G)\cap \mathbb{L}_{2}(G)$ 
\end{definition}

From the definition of $\mathbb{L}_{1}(\G)$ and $\mathbb{L}_{2}(G)$, the pencil $\lambda \mathbb{X}+ \mathbb{Y} = \lambda \left[                                                              \begin{array}{c|c}                                                                 X_{11} & X_{12} \\\hline                          
                               X_{21} & X_{22} \\
                             \end{array}\right] +
                 \left[\begin{array}{c|c}
                      Y_{11} & Y_{12} \\\hline
                      Y_{21} & Y_{22}\\
                 \end{array}\right]$
 belonging to $\mathbb{DL}(G)$ should satisfy the following: $X_{12} = 0$, $X_{21} = 0$ and $Y_{21} =e_{k}s^{T} \otimes C$. But $Y_{21} = w e_{m}^{T}\otimes C$. So, $e_{k} s^{T} = w e_{m}^{T}$. Similarly, $ Y_{12} = -v e_{k}^{T} \otimes B$. But $Y_{12} = -e_m z^{T}\otimes B$ which implies that $v e_{k}^{T} = e_{m}z^{T}$. Thus we get $w  = \beta e_{k},  z= \alpha e_k  \mbox{  and  }  s  = \beta e_{m},  v =\alpha e_m$ for scalars $\alpha$ and $\beta.$
Now, define
\begin{align*}
&\mathbb{DL}(G) 
  = \biggl\{\lambda 
     \left[\begin{array}{c|c}
          X_{11} & 0 \\\hline
          0 & X_{22} \\
     \end{array}\right] + 
     \left[\begin{array}{c|c}
          Y_{11} & -v e_{k}^{T} \otimes B \\\hline
          we_{m}^{T} \otimes C & Y_{22} \\
     \end{array}\right] :  
  \lambda X_{11}+Y_{11} \in  \mathcal{DL}(A),  \\
& \hspace{2.5cm} \lambda X_{22}+Y_{22} \in  \mathcal{DL}(D)
      \mbox{  with  } w  = \beta e_{k},  z= \alpha e_k  \mbox{  and  }  s  = \beta e_{m},  v =\alpha e_m \biggr\},
\end{align*}
where $ \mathcal{DL}(A)$ is the subspace of all block-symmetric pencils in $\mathcal{L}_{1}(A)$ \cite{mmmm06}. 
OR, 
\begin{align*}
&\mathbb{DL}(G) 
  = \biggl\{ 
      \left[\begin{array}{c|c}
         L(\lambda) & -v e_{k}^{T} \otimes B  \\\hline
         we_{m}^{T} \otimes C & K(\lambda) \\
      \end{array}\right] :
    L(\lambda) \in  \mathcal{DL}(A),  \\ 
    & \hspace{2.5cm} K(\lambda) \in  \mathcal{DL}(D)
      \mbox{  with  } w  = \beta e_{k},  z= \alpha e_k  \mbox{  and  }  s  = \beta e_{m},  v =\alpha e_m\biggr\}.
\end{align*}

Let $ \mathcal{H} := \sum^k_{i=1}\sum^l_{j=1} E_{ij}\otimes H_{ij}$ be a  $k \times l$ block matrix with each block $H_{ij}$ being a $p\times q$ matrix, where $ E_{ij}$ is a $k\times l$ matrix whose $(i,j)$-th entry is equal to $1$ and the remaining entries are equal to $0.$ The block transpose of $\mathcal{H}$ denoted by  $\mathcal{H}^{\mathcal{B}} $ is a $l \times k$ block matrix given by
 $\mathcal{H}^{\mathcal{B}}:= \sum^k_{i=1}\sum^l_{j=1} E_{ij}^T\otimes H_{ij}.$ A block matrix $\mathcal{H}$ is said to be {\em block-symmetric} provided that $\mathcal{H}^{\mathcal{B}} = \mathcal{H},$ see~\cite{hmmt}.

The double ansatz space $ \mathcal{DL}(A) = \mathcal{L}_1(A) \cap \mathcal{L}_2(A)$ consists of block-symmetric pencils. More precisely,
if $L(\lambda)\in   \mathcal{DL}(A)$ with right ansatz vector $v$ and left ansatz vector $w$ then $v = w$ and $L(\lambda)$ is block-symmetric. In particular, $ \mathcal{DL}(A) $ contains symmetric (resp., Hermitian) linearizations of $A(\lam)$ when $A(\lam)$ is symmetric (resp., Hermitian) and regular,  see~\cite{hmmt, mmmm06}.

We  define the block-transpose of a system matrix as follows.

\begin{definition}
Let $\mathcal{A}:=\left[\begin{array}{c|c}
          A & u s^{T} \otimes X \\\hline
          z v^{T} \otimes Y & D 
          \end{array} \right] 
    \in \mathbb{C}^{(mn+kr) \times (mn+kr) },$
where $A=[A_{ij}]$ is $m \times m$ block matrix with $A_{ij}\in \mathbb{C}^{n \times n}$, $u, v \in \mathbb{C}^{m}$, $X \in \mathbb{C}^{n \times r}$, $Y \in \mathbb{C}^{r \times n}$ and $D=[D_{ij}]$ is $k \times k$ block matrix with $D_{ij}\in \mathbb{C}^{r \times r}$. Define the block transpose of $\mathcal{A}$ by 
$\mathcal{A}^{\mathcal{B}}
    := \left[\begin{array}{c|c}
          A^{\mathcal{B}} & v z^{T} \otimes X \\\hline
          s u^{T} \otimes Y &  D^{\mathcal{B}} \end{array} 
      \right].$
\end{definition}

Observe that $\mathcal{A}$ is block-symmetric if and only if $A$, $D$  are block-symmetric and $u = v$, $z=s$. Consider the set of double ansatz pencils $\mathbb{DL}(G) := \mathbb{L}_{1}(G) \cap \mathbb{L}_{2}(G)$. 
%We write $(\mathbb{L}_k(\lam), v_k, w_k) \in \mathbb{L}_k(G)$ to mean that $\mathbb{L}_k(\lam) := \mathbb{A}_k(L_k(\lam), v_k, w_k)$ for $k =1, 2.$ 
Let $(\mathbb{L}(\lambda), v, w)  \in \mathbb{L}_{1}(G)$ be given by  

$$\mathbb{L}(\lambda) 
    = \left[\begin{array}{c|c}
          L(\lam) & -ve_{k}^{T} \otimes B \\\hline
          we_{m}^{T} \otimes C & K(\lam) \\
      \end{array}\right] \in \mathbb{L}_{1}(G)$$
with $(L(\lam), v) \in \mathcal{L}_{1}(A)$ and $(L(\lam), w) \in \mathcal{L}_{1}(D)$. 
%                             
%Now consider the set of double ansatz pencils $\mathbb{DL}(G):=\mathbb{L}_{1}(G) \cap \mathbb{L}_{2}(G).$    Let $ (\mathbb{T}(\lam), v) \in
 % \mathbb{L}_1(G)$ be given by $$\mathbb{T}(\lambda)=\left[
 %                           \begin{array}{c|c}
 %          L(\lambda) & -v \otimes B \\
 %                             \hline
%          e_m^{T} \otimes C & D
%                            \end{array}
%                 \right], \mbox{ where }
 %                (L(\lambda),v) \in\mathcal{L}_1(P).$$
Then it is easily seen that $(\mathbb{L}(\lambda),v, w)$ $\in \mathbb{L}_1(G) \iff (\mathbb{L}(\lambda)^{\mathcal{B}},v, w) \in \mathbb{L}_2(G)$  whenever $(L(\lambda), v) \in  \mathcal{DL}(A)$ and $(K(\lambda), w) \in  \mathcal{DL}(D).$ Consequently, we have $(\mathbb{L}(\lambda),v, w) \in \mathbb{DL}(G) \iff (L(\lam), v) \in \mathcal{DL}(P)$, $(K(\lambda), w) \in \mathcal{DL}(D)$ and $ v= e_m$, $w=e_k$.  Since an ansatz vector $v$ uniquely determines~\cite{mmmm06} a pencil in $ \mathcal{DL}(A)$, we conclude that $\mathbb{DL}(G)$ consist of a single pencil $(\mathbb{L}(\lam), e_m, e_k)$ given by

$$ \mathbb{L}(\lam) 
       = \left[ \begin{array}{c|c}
               L(\lambda) & -e_me_{k}^{T} \otimes B \\\hline
               e_k e_m^{T} \otimes C & K(\lambda)
         \end{array} \right],$$
where $(L(\lambda), e_m) \in \mathcal{DL}(A)$, $(K(\lambda), e_k) \in  \mathcal{DL}(D).$ We refer to $\mathbb{L}(\lam) \in \mathbb{DL}(G)$ as a double ansatz pencil of $G(\lam)$ as well as of $\mathcal{S}(\lam).$

\begin{definition}[\cite{vardulakis}]
Let $\mathcal{S}_1 (\lambda)$ and $\mathcal{S}_2 (\lambda)$ be system matrices with extended forms,
	
 \[\mathcal{S}^{e}_{i} (\lambda) 
      = \left[\begin{array}{cc|c} 
                I_{q-r_i} & 0 & 0 \\ 0 & A_i(\lambda) & B_i (\lambda) \\ \hline 0 & C_i (\lambda) & D_i(\lambda)
	  \end{array}\right] 
 \]
where $A_i(\lambda)\in \mathbb{C}[\lambda]^{r_i \times r_i}$, $B_i(\lambda)\in \mathbb{C}[\lambda]^{r_i \times m}$, $C_i(\lambda) \in \mathbb{C}^{p\times r_i}$, $D(\lambda) \in \mathbb{C}^{p\times m}$, $n_i = deg|A_i(\lambda)|$ and $q\geq max\{n_1, n_2\}$. Then, $\mathcal{S}_1 (\lambda)$ and $\mathcal{S}_2 (\lambda)$ are said to be strong system equivalent (SSE) if the exist unimodular matrices $U(\lambda)$, $V(\lambda) \in \mathbb{C}[\lambda]^{q \times q}$ and $X(\lambda)\in \mathbb{C}[\lambda]^{p\times q}$, $Y(\lambda) \in \mathbb{C}[\lambda]^{q\times m}$ such that,

{\footnotesize
$$
  \left[\begin{array}{cc} 
         U(\lambda) & 0 \\ 
         X(\lambda) & I_p
  \end{array}\right] 
  \left[\begin{array}{ccc} 
         I_{q-r_1} & 0 & 0 \\ 0 & A_1(\lambda) & B_1 (\lambda) \\ 
         0 & C_1 (\lambda) & D_1(\lambda) 
  \end{array}\right] 
= \left[\begin{array}{ccc} 
         I_{q-r_2} & 0 & 0 \\
         0 & A_2(\lambda) & B_2 (\lambda) \\
         0 & C_2 (\lambda) & D_2(\lambda)
  \end{array}\right] 
  \left[ \begin{array}{cc}
         V(\lambda) & Y(\lambda)\\ 
         0 & I_m 
  \end{array} \right]. 
$$} 
\end{definition}

\begin{definition}[System Linearization]
A pencil of the form {\footnotesize $\mathbb{L} = \left[ \begin{array}{c|c} L(\lambda) & \mathcal{B} \\ \hline \mathcal{C} & K(\lambda) \end{array}\right]$} {\small $\in \mathbb{C}^{(nm+r)\times (nm+r)}$}, where $L(\lambda)$ is an $nm \times nm $ and $K(\lambda)$ is an $rk \times rk $ matrix pencil, is called a system linearization of $\mathcal{S}(\lambda)$ if $\mathcal{S}(\lambda)$ is strict system equivalent to $\mathbb{L}(\lambda)$.
\end{definition}

\begin{definition}[Rosenbrock Linearization]
Let $\mathcal{S} (\lambda)$ be the system matrix and $m = deg(A) \geq 2$ and $k = deg(D) \geq 2$. Then an $(mn+kr) \times (mn+kr)$ system matrix $\mathbb{L} (\lambda)$ given by 
		
$$
 \mathbb{L} (\lambda) 
     = \left[\begin{array}{c|c} 
            L(\lambda) & \mathcal{C} \\ \hline
			\mathcal{B} & K(\lambda) 
       \end{array}\right],
$$  
where $L(\lambda)= X- \lambda Y$, $K(\lambda) = H - \lambda K$  with $K$ being non-singular, is said to be a Rosenbrock linearization of $\mathcal{S} (\lambda)$ provided that there are $mn\times mn$ unimodular matrix polynomials $U(\lambda)$ and $V(\lambda)$ and $kr\times kr$ unimodular  matrix polynomials $\widetilde{U}(\lam)$ and $\widetilde{V}(\lam)$ such that

$$
\left[\begin{array}{c|c} 
        U(\lambda) &  \\ \hline
		&  \tilde{U}(\lam) 
\end{array} \right]
\mathbb{L}(\lambda)
\left[\begin{array}{c|c} 
       V(\lambda) &  \\  \hline
       &  \tilde{V}(\lam) 
       \end{array}\right] 
= \left[\begin{array}{c|c|c} 
       I_{(m-1)n} &  &  \\    \hline 
       &  \mathcal{S} (\lambda) & \\
       \hline 
       &  &  I_{(k-1)r} 
   \end{array}  \right]
$$
for all $\lambda \in \mathbb{C}$. 
\end{definition} 
		
%\begin{remark}
%If $\mathbb{L}(\lambda) $ is a Rosenbrock linearization of $\mathcal{S}(\lambda)$ , then $\Sigma_L$ is controllable(resp. observable) if and only if $\Sigma$ is controllable(resp. observable). Moreover, the LTI systems $\Sigma$ and $\Sigma_L$ have the same finite zeroes(transmission zeroes, invariant zeroes, input/output decoupling zeroes) and same finite poles. Further, for all $\lambda \in \mathbb{C}$, $ U(\lambda) (\mathcal{X} - \lambda \mathcal{Y}) V(\lambda) = diag(I_{(m-1)n}, P(\lambda)) $ and $U(\lambda) \mathbb{G} (\lambda) V(\lambda) = diag(I_{(m-1)n}, G(\lambda)).$
%\end{remark}
	
\begin{theorem}\label{mackey01}	
Let $A(\lambda) = \sum_{j=0}^{m} {\lambda^i A_j}$ with $A_m \ne 0$ be an $ n\times n $ matrix polynomial (regular or singular) and $ L(\lambda) = \lambda X + Y $ be a matrix pencil.
\begin{enumerate}[label=(\alph*)]			
	\item If $L(\lambda) \in \mathcal{L}_1 (A) $ and has right ansatz vector 
          $e_1$, then, 
		 $$X = \left[\begin{array}{c|c} 
                       A_m & X_{12} \\ \hline
                       0 & -Z	
                 \end{array}\right] \qquad 
		  Y = \left[\begin{array}{c|c}
                       Y_{11} & A_0 \\ \hline
                       Z & 0 
                \end{array}\right],$$	
		 where $ X_{12}$, $Y_{11} \in \mathbb{C}^{n \times (m-1)n} $ such that {\small $ X_{12} + Y_{11} = \left[\begin{array}{cccc} A_{m-1} & A_{m-2} & \cdots & A_{1} \end{array}\right] $} and $ Z \in \mathbb{C}^{(m-1)n \times (m-1)n}$.
   
    \item If $L(\lambda) \in \mathcal{L}_1 (A)$ and has the right ansatz vector 
          $v \ne 0$, then there exists a non-singular matrix $M \in \mathbb{C}^{m\times m}$ such that $Mv = e_1$ and the pencil $(M\otimes I_n) L(\lambda)$ is given by 
		 $$(M\otimes I_n) L(\lambda) 
              = \lambda \left[ \begin{array}{c|c} 
                            A_m & X_{12} \\ \hline 
                            0 & -Z 
                        \end{array} \right]  
                + \left[ \begin{array}{c|c} 
                        Y_{11} & A_0 \\ \hline
                        Z & 0 
                  \end{array} \right]$$	 
          for some $Z \in \mathbb{C}^{(m-1)n \times (m-1)n}$ and $X_{12}, Y_{11}$ are as in $(a)$.
          
	\item If $L(\lambda) \in \mathcal{L}_2 (A)$ and has left ansatz vector $w \ne      0$, then there exists a non-singular matrix $K \in \mathbb{C}^{m\times 
          m}$ such that $K^T w = e_1$ and the pencil $ L(\lambda) (K \otimes I_n)$ is given by 
		 $$L(\lambda) (K\otimes I_n) 
             = \lambda \left[\begin{array}{c|c} 
                           A_m & 0 \\ \hline 
                           X_{21} & -Z 
                       \end{array} \right] 
               + \left[\begin{array}{c|c} 
                           Y_{11} & Z \\ \hline 
                           A_0 & 0 
                      \end{array}\right],$$ 
         where $X_{21}$, $Y_{11} \in \mathbb{C}^{(m-1)n \times n}$ are such that $ X_{21} + Y_{11} = \left[\begin{array}{ccc} A^{T}_m & \cdots & A^{T}_1  \end{array}\right] $ and $ Z \in \mathbb{C}^{(m-1)n\times (m-1)n}$.
		
	\end{enumerate}
\end{theorem}

All the matrices that appear in the block labelled $Z$ have the same rank. Hence, for $L(\lambda) \in \mathcal{L}_1 (A)$ or $L(\lambda) \in \mathcal{L}_2 (A)$ it makes sense to talk about the Z-rank of $L(\lambda)$.

\begin{definition}
	The \textit{Z-rank} of $L(\lambda) \in \mathcal{L}_1 (A)$ is the rank of any matrix in the block labelled $Z$ in Theorem~\ref{mackey01} under any reduction of $L(\lambda)$ of the form therein. If $Z$ is non-singular, we say that $L(\lambda) \in \mathcal{L}_1 (A)$ has full $Z-\rank$.
	
	Similarly, the $Z-\rank$ of $L(\lambda) \in \mathcal{L}_2 (A)$ is the rank of any matrix appearing in the block labelled $Z$. 	
\end{definition}

Next two theorems show that almost all the pencils in $\mathbb{L}_1 (G)$ and $\mathbb{L}_2 (G)$ are linearizations of $\mathcal{S}(\lambda)$.

\begin{theorem}\label{lin01}
Let 
  $\mathbb{L}(\lambda) 
    = \left[ \begin{array}{c|c} 
          L(\lambda) & -ve_{k}^{T}\otimes B \\ \hline 
          we^T _m \otimes C & K(\lambda)  
      \end{array} \right] 
      \in \mathbb{L}_{1} (G)$, 
where $(L(\lambda), v) \in \mathcal{L}_1 (A)$, $(K(\lambda), w) \in \mathcal{L}_1 (D)$ and $v \ne 0$, $w \ne 0$. If $L(\lambda)$ has full Z-rank then there exist $mn\times mn$ unimodular matrix polynomials $U(\lambda)$ and $V(\lambda)$ such that $ U(\lambda)^{-1} (e_m \otimes I_n) = v\otimes I_n $ and $(e^{T}_m \otimes I_n) V(\lambda)^{-1} = e^{T}_m \otimes I_n$. Also, if  $K(\lambda)$ has full Z-rank then there exist $kr\times kr$ unimodular matrix polynomials $\widetilde{U}(\lambda)$ and $\widetilde{V}(\lambda)$ such that $ \widetilde{U}(\lambda)^{-1} (e_k\otimes I_r) = w\otimes I_r $ and $(e^{T}_k \otimes I_r) \widetilde{V}(\lambda)^{-1} = e^{T}_k \otimes I_r$. Further, we have $\diag\left(U(\lambda), \widetilde{U}(\lambda)\right) \mathbb{L}(\lambda) \diag\left(V(\lambda), \widetilde{V}(\lambda)\right) = \diag\left(I_{(m-1)n}, \mathcal{S}(\lambda), I_{(k-1)r}\right) $ for all $\lambda \in \mathbb{C}$. Thus $\mathbb{L}(\lambda)$ is a Rosenbrock linearization of $\mathcal{S}(\lambda)$.
\end{theorem}
	
\begin{proof}
There exists a non-singular $M\in \mathbb{C}^{m\times m}$ such that $Mv=e_1$. Then, ($M\otimes I_n)L(\lambda) \in \mathcal{L}_1(A)$ with right ansatz vector $e_1$. Hence by theorem (\ref{mackey01}), we have 

$$
(M\otimes I_n) L(\lambda) 
    = \lambda \left[\begin{array}{c|c} 
                 A_m & X_{12} \\ \hline 
                 0 & -Z 
              \end{array} \right] 
    + \left[\begin{array}{c|c} 
          Y_{11} & A_0 \\ \hline 
          Z & 0 
      \end{array}\right]
$$      
with non-singular matrix $Z\in \mathbb{C}^{(m-1)n \times (m-1)n}.$ Let $\widehat{L}(\lambda) = (M\otimes I_n) L(\lambda)$. Then there exist unimodular matrix polynomials $\widehat{U}(\lambda)$ and $\widehat{V}(\lambda)$ such that,

$$
\widehat{U}(\lambda)\widehat{L}(\lambda)\widehat{V}(\lambda) 
     = \widehat{U}(\lambda) (M\otimes I_n ) L(\lambda) \widehat{V}(\lambda)  
     = \left[\begin{array}{c|c} 
           I_{(m-1)n} &   \\ \hline
           & A(\lambda) 
      \end{array} \right] 
$$
and $\widehat{U}(\lambda)^{-1}(e_m \otimes I_n) = e_1 \otimes I_n$ and $ (e^T_m \otimes I_n ) \widehat{V}(\lambda)^{-1} = e^T_m \otimes I_n$. Consequently, we have, 
$\left(\widehat{U}(\lambda)(M\otimes I_n)\right)^{-1}\left(e_m \otimes I_n\right)= \left(M^{-1}\otimes I_n\right)\widehat{U}(\lambda)^{-1}= v\otimes I_n$, see proof of Theorem 4.4 of \cite{DA19}. Similarly, there exists a non-singular $N\in \mathbb{C}^{k\times k}$ such that $Nw=e_1$. Then, ($N\otimes I_r)K(\lambda) \in \mathcal{L}_1(D) $ with right ansatz vector $e_1$. Hence by theorem (\ref{mackey01}), we have 

$$
(N\otimes I_r) K(\lambda) 
   = \lambda \left[\begin{array}{c|c} 
                D_k & S_{12} \\ \hline 
                0 & -Z_{1} 
             \end{array} \right] 
     + \left[\begin{array}{c|c} 
           H_{11} & D_0 \\ \hline 
           Z_1 & 0 
        \end{array} \right]
$$
with non-singular matrix $Z\in \mathbb{C}^{(k-1)r \times (k-1)r}. $ Let $\widehat{K}(\lambda) = (N\otimes I_r) K(\lambda)$. Then there exist unimodular matrix polynomials $\widehat{P}(\lambda)$ and $\widehat{Q}(\lambda)$ such that,

$$
\widehat{P}(\lambda) \widehat{K}(\lambda) \widehat{Q}(\lambda) 
      = \widehat{P}(\lambda) (N\otimes I_r ) K(\lambda) \widehat{Q}(\lambda)  
      = \left[\begin{array}{c|c} 
              D(\lambda) &   \\ \hline
              & I_{(k-1)r} 
        \end{array} \right] 
$$
and $\widehat{P}(\lambda)^{-1}(e_k \otimes I_r) = e_1 \otimes I_r$ and $ (e^T_k \otimes I_r ) \widehat{Q}(\lambda)^{-1} = e^T_k \otimes I_r$. Consequently, we have, 
$(\widehat{P}(\lambda)(N\otimes I_r))^{-1}(e_k \otimes I_r)= (N^{-1}\otimes I_r)\widehat{P}(\lambda)^{-1}= w\otimes I_r$.
By setting $U(\lambda) = \widehat{U}(\lambda)(M\otimes I_n)$ and $V(\lambda) = \widehat{V}(\lambda)$, and $\widetilde{U}(\lambda) = \widehat{P}(\lambda)(N\otimes I_r)$ and $\widetilde{V}(\lambda) = \widehat{Q}(\lambda)$ we have 

 {\footnotesize
\begin{align*}
   &\left[\begin{array}{c|c} 
            U(\lambda) & 0 \\ \hline 
            0 & \widetilde{U}(\lambda) 
    \end{array}\right] 
    \mathbb{L}(\lambda) 
    \left[\begin{array}{c|c} 
           V(\lambda) & 0 \\ \hline
           0 & \widetilde{V}(\lambda) 
    \end{array}\right] \\ 
  &=\left[\begin{array}{c|c} 
          \widehat{U}(\lambda)(M\otimes I_n) & 0 \\ \hline
          0 & \widehat{P}(\lambda)(N\otimes I_r) 
          \end{array}\right] 
    \left[\begin{array}{c|c} 
          L(\lambda) & -ve_{k}^{T}\otimes B \\ \hline 
          we^T_m \otimes C & K(\lambda)
          \end{array}\right] 
    \left[ \begin{array}{c|c} 
         \widehat{V}(\lambda) & 0 \\ \hline 
         0 &  \widehat{Q}(\lambda) 
    \end{array} \right] \\ 
  &=\left[\begin{array}{c|c} 
        \widehat{U}(\lambda)(M\otimes I_n) & 0 \\ \hline
        0 & \widehat{P}(\lambda)(N\otimes I_r)
    \end{array} \right]\times \\ 
  & \quad\left[\begin{array}{c|c} 
        L(\lambda) & (\widehat{U}(\lambda)(M\otimes I_n))^{-1} (-e_m \otimes B)(e^T_k \otimes I_r ) \widehat{Q}(\lambda)^{-1} \\ \hline 
        (\widehat{P}(\lambda)(N\otimes I_r))^{-1}(e_k \otimes I_r)(e^T_m \otimes C)\widehat{V}^{-1} & K(\lambda) 
    \end{array} \right] \\
  & \quad\times\left[\begin{array}{c|c}
       \widehat{V}(\lambda) & 0 \\ \hline
       0 & \widehat{Q}(\lambda) 
    \end{array} \right]\\ 
  &=\left[\begin{array}{c|c} 
       \widehat{U}(\lambda)(M\otimes I_n)L(\lambda) \widehat{V}(\lambda) & -ve_{k}^{T} \otimes B \\ \hline 
       we^T_m \otimes C & \widehat{P}(\lambda)(N\otimes I_r) K(\lambda)\widehat{Q}(\lambda) 
     \end{array}\right] \\ 
  &=\left[\begin{array}{cc|cc}
       I_{(m-1)n} &  & -ve_{k}^{T} \otimes B & \\
       & A(\lambda) & &  \\ \hline
       we^T_m \otimes C &   & D(\lambda) &  \\
       & & & I_{(k-1)r }
    \end{array} \right].
\end{align*}}
\end{proof}

By Lemma~\ref{L1L2r}, we know that $\mathbb{L}_{2} (G) = (\mathbb{L}_{1}(G^{T}))^{T}.$ Thus we have the following theorem. 
 	
\begin{theorem}
Let 
$\mathbb{L}(\lambda)
    =\left[ \begin{array}{c|c} 
          L(\lambda) & -e_m z^T  \otimes B \\ \hline 
          e_k s^{T}\otimes C & K(\lambda) 
     \end{array}\right] 
     \in \mathbb{L}_2 (G)$, 
where $(L(\lambda), z)$ $\in \mathcal{L}_2 (A), (K(\lam), s) \in \mathcal{L}_2 (D)$ and $s, z \ne 0$. If $L(\lambda)$ has full Z-rank then there exist $mn\times mn$ unimodular matrix polynomials $U(\lambda)$ and $V(\lambda)$ such that $ U(\lambda)^{-1} (e_m \otimes I_n) = e_m \otimes I_n $ and $ (e^T _m \otimes I_n) V(\lambda)^{-1} = z^T \otimes I_n$. Also, if $K(\lambda)$ has full Z-rank then there exist $kr\times kr$ unimodular matrix polynomials $\Tilde{U}(\lambda)$ and $\Tilde{V}(\lambda)$ such that $ \Tilde{U}(\lambda)^{-1} (e_k \otimes I_r) = e_k \otimes I_r$ and $(e^T _k \otimes I_r) \Tilde{V}(\lambda)^{-1} = s^T \otimes I_r$. Further, we have $\diag\left(U(\lambda), \Tilde{U}(\lambda)) \mathbb{L} (\lambda) \diag(V(\lambda), \Tilde{V}(\lambda)\right) = \diag\left(I_{(m-1)n}, \mathcal{S}(\lambda), I_{(k-1)r}\right)$ for all $\lambda \in \mathbb{C}$. Thus $\mathbb{L}(\lambda)$ is a Rosenbrock linearization of $\mathcal{S}(\lambda)$.
\end{theorem}

\begin{proof}
	Let $Q\in \mathbb{C}^{m\times m}$ and $R \in \mathbb{C}^{k\times k}$ be a non-singular matrix such that $z^T Q = e^T_1$. Then, by theorem \ref{mackey01}, there is a non-singular matrix $Z\in \mathbb{C}^{(m-1)n \times (m-1)n}$ and $Z_1\in \mathbb{C}^{(k-1)r \times (k-1)r}$ such that,
\[ L(\lambda)(Q\otimes I_n) = \lambda \left[ \begin{array}{c|c} A_m & 0 \\ \hline X_{21} & -Z \end{array} \right] + \left[ \begin{array}{c|c} Y_{11} & Z \\ \hline A_0 & 0 \end{array} \right] \]
and 
\[ K(\lambda)(R\otimes I_r) = \lambda \left[ \begin{array}{c|c} D_k & 0 \\ \hline S_{21} & -Z_1 \end{array} \right] + \left[ \begin{array}{c|c} H_{11} & Z_1 \\ \hline D_0 & 0 \end{array} \right] \]
Hence we have
\begin{equation} \label{5}
(Q^T \otimes I_n) L(\lambda)^T = \lambda \left[ \begin{array}{c|c} A^T_m &  X^T_{21}\\ \hline 0  & -Z^T \end{array} \right] + \left[ \begin{array}{c|c} Y^T_{11} & A^T_0 \\ \hline Z^T & 0 \end{array} \right]     
\end{equation}  and 
\begin{equation} \label{6}
(R^T \otimes I_r) K(\lambda)^T = \lambda \left[ \begin{array}{c|c} D^T_k & S^T_{21}  \\ \hline 0 & -Z_1^T \end{array} \right] + \left[ \begin{array}{c|c} H^T_{11} & D^T_0  \\ \hline Z_{1}^T & 0 \end{array} \right].    
\end{equation} 
Since $L(\lambda) \in \mathcal{L}_2 (A)$ and $K(\lambda) \in \mathcal{L}_2 (D)$, by (\ref{5}) and (\ref{6}), we have $L(\lambda)^T \in \mathcal{L}_2 (A)$ and $K(\lambda)^T \in \mathcal{L}_2 (D)$ with full $Z$-rank with 
$\left[ \begin{array}{c} 
z \\
s \end{array} \right]$ as the right ansatz vector. Then by theorem \ref{lin01}, there exist unimodular matrix polynomials $\hat{U}(\lambda)$,  $\hat{V}(\lambda)$ and $\widehat{P}(\lambda)$,  $\widehat{Q}(\lambda)$ such that, such that,
	\[ \widehat{U}(\lambda) L(\lambda)^T \widehat{V}(\lambda) = \left[ \begin{array}{c|c} I_{(m-1)n} &  \\ \hline    & A(\lambda)^T \end{array} \right] \implies \widehat{V}(\lambda)^T L(\lambda) \widehat{U}(\lambda)^T = \left[ \begin{array}{c|c} I_{(m-1)n} &  \\ \hline    & A(\lambda) \end{array} \right] \]  and 
 \[ \widehat{P}(\lambda)K(\lambda)^{T} \widehat{Q}(\lambda) = \left[ \begin{array}{c|c} D(\lambda)^{T} &   \\ \hline    & I_{(k-1)r} \end{array} \right] =  \widehat{Q}(\lambda)^{T}  K(\lambda) \widehat{P}(\lambda)^{T} = \left[ \begin{array}{c|c} D(\lambda) &   \\ \hline    & I_{(k-1)r} \end{array} \right] \]
Also, by theorem \ref{lin01} we have  $\left(\widehat{V}(\lambda)\right)^{-T} (e_m \otimes I_n) = e_m \otimes I_n $ and $(e^T_m \otimes I_n)\left(\widehat{U}(\lambda)\right)^{-T}= s^T \otimes I_n$ 
 and $\widehat{Q}(\lambda)^{-1}(e_k \otimes I_r) = e_k \otimes I_r$ and $ (e^T_k \otimes I_r ) \widehat{P}(\lambda)^{-1} = z^T \otimes I_r$. By setting $U(\lambda) =\widehat{V}(\lambda)^T$ and $V(\lambda)= \widehat{U}(\lambda)^T$, $\Tilde{U}(\lambda) =\widehat{Q}(\lambda)^T$ and $\Tilde{V}(\lambda)= \widehat{P}(\lambda)^T$ the desired result follows.
\end{proof}

\subsection{Hermitian Linearization}
Consider $\mathcal{S}(\lam)$ and $G(\lam)$ are given in (\ref{sysmatrix01}) and (\ref{tfunction01}), respectively.
We now construct the Hermitian (resp., symmetric) linearization of $G(\lam)$  and $\mathcal{S}(\lam)$ when $G(\lam)$  and $\mathcal{S}(\lam)$ are Hermitian (resp., symmetric).
We define the adjoint of $\mathcal{S}(\lam)$  by
$$\mathcal{S}^{*}(\lam) = \left[
                                         \begin{array}{c|c}
                                           A^{*}(\lam) & C^{*} \\
                                           \hline
                                           -B^{*} & D^{*}(\lam) \\
                                         \end{array}
                                       \right],
$$ where $A^{*}(\lam) = \sum_{j = 1}^{m}\lam^{j}A_j^{*}$. The transposes of $\mathcal{S}(\lam)$ and $A(\lam)$ are defined similarly. The system matrix $\mathcal{S}(\lam)$ is said to be Hermitian if  $\mathcal{S}^*(\lam) = \mathcal{S}(\lam)$ for $\lam \in \C$. The transfer function  $G(\lam)$  is said to be Hermitian (resp., symmetric) when $\mathcal{S}(\lam)$ is
Hermitian (resp., symmetric).

%{\bf Definition. (Transpose and Adjoint of Matrix Rational Functions).}
%Let $\textbf{R}(\lambda) = \sum_{i=0}^{m}A_{i}\lambda^{i}+L(C-\lambda %D)^{-1}U$ be a matrix rational functions of degree $n$. Then
%\begin{equation}
%\textbf{R}^{T}(\lambda) := \sum_{i=0}^{m}A_{i}^{T}\lambda^{i}+B^{T}(C^{T}-\lambda D^{T})^{-1}A^{T} \mbox{  and  } \textbf{R}^{*}(\lambda) := \sum_{i=0}^{m}A_{i}^{*}\lambda^{i}+B^{*}(C^{*}-\lambda D^{*})^{-1}A^{*}
%\end{equation} defines the transpose and adjoint of $\textbf{R}(\lambda)$.

\begin{center}
Basic definition of structures of matrix rational functions $G(\lambda) = D(\lambda) + C A(\lambda)^{-1}B.$
\end{center}
\begin{center}
\begin{tabular}{|r|r|r|c|c|c|}
  \hline
  Symmetric & $\mathcal{S}^{T}(\lambda) = \mathcal{S}(\lambda)$ & $A_{i}^{T} = A_{i}, \mbox{ and }  C^{T} = B, D_{i}^{T} = D_{i}$  \\
  \hline
  Skew Symmetric & $\mathcal{S}^{T}(\lambda) = -\mathcal{S}(\lambda)$ & $A_{i}^{T} = -A_{i} \mbox{ and }, C^{T} = -B, D_{i}^{T} = -D_{i} $\\ \hline
  Hermitian & $\mathcal{S}^{*}(\lambda) =\mathcal{S}(\bar{\lambda})$ & $A_{i}^{*} = A_{i}, \mbox{ and }  C^{*} = B, D_{i}^{*} = D_{i}$  \\
      \hline
\end{tabular}
\end{center}

Define $$\mathbb{S}(G) := \left\{\lambda \mathbb{X}+ \mathbb{Y} \in \mathbb{L}_{1}(G) : \mathbb{X}^{T} = \mathbb{X}, \mathbb{Y}^{T} = \mathbb{Y} \right\}$$ the set of all symmetric pencils in $\mathbb{L}_{1}(G)$, when $G$ is symmetric. i.e.,
$$\lambda \mathbb{X} +\mathbb{Y} = \lambda \left[
                           \begin{array}{c|c}
                             X_{11} & X_{12} \\
                             \hline
                             X_{21} & X_{22} \\
                           \end{array}
                         \right]+ \left[
                                    \begin{array}{c|c}
                                      Y_{11} & Y_{12} \\
                                      \hline
                                      Y_{21} & Y_{22} \\
                                    \end{array}
                                  \right] \in \mathbb{L}_{1}(G) \mbox{ and } \mathbb{X}^{T} = \mathbb{X}, \mathbb{Y}^{T} = \mathbb{Y}
$$ So,
$$ X_{21} = 0 \Rightarrow X_{21}^{T} = X_{12} = 0, \mbox{ and }  Y_{21} = Y_{12}^{T} = w e_{m}^{T} \otimes C = w e_{m}^{T} \otimes B^{T},$$
$$ 0 \boxplus Y_{12} = -ve_{k+1}^{T} \otimes B \Rightarrow Y_{21} = Y_{12}^{T} = -e_{k}v^{T}  \otimes B^{T} = w e_{m}^{T} \otimes B^{T}$$
Therefore, we get
\begin{center}
\begin{tabular}{|r|}
  \hline
  $  -e_{k}v^{T} = w e_{m}^{T} \implies v=e_m \text {  and   }   w=e_k. $ \\
  \hline
\end{tabular}
\end{center}
%$$ v = -e_{d}$$
So,
\begin{align*}
&\mathbb{S}(G) = \biggr\{\lambda \left[
                                          \begin{array}{c|c}
                                            X_{11} & 0 \\
                                            \hline
                                            0 & X_{22} \\
                                          \end{array}
                                        \right]+ \left[
                                                   \begin{array}{c|c}
                                                     Y_{11} & e_{m}e_{k}^{T}\otimes C^{T} \\
                                                     \hline
                                                     e_{k}e_{m}^{T} \otimes C &Y_{22} \\
                                                   \end{array}
                                                 \right] : 
X_{11}^{T} = X_{11}, \,\, X_{22}^{T} = X_{22},\,\,\,  Y_{11}^{T} = Y_{11},  \\ & \hspace{1.7cm} Y_{22}^{T} = Y_{22}, \,\,\, X_{11}\boxplus Y_{11} = e_m \otimes [A_{m} \ldots A_{0}], \,\,\, X_{22}\boxplus Y_{22} = e_k \otimes [D_{k} \ldots D_{0}]. 
\biggr\}
\end{align*}

\begin{exam}
Let $G(\lambda) =  A_{3}\lambda^{3}+A_{2}\lambda^{2}+A_{1}\lambda+A_{0}+C(D_0-\lambda D_1)^{-1}B$ be a symmetric matrix rational function, then a symmetric linearization is given by
$$\left( \lambda \left[
                                \begin{array}{ccc|c}
                                  0 & 0& A_{3} & 0 \\
                                  0 & A_{3} & A_{2} & 0 \\
                                  A_{3} & A_{2} & A_{1} & 0\\
                                  \hline
                                  0 & 0 & 0 & -D_1\\
                                \end{array}
                              \right] + \left[
                      \begin{array}{ccc|c}
                        0 & -A_{3} & 0 & 0 \\
                        -A_{3} & -A_{2} & 0 & 0 \\
                        0 & 0 & A_{0} & B \\
                        \hline
                        0 & 0 & B^{T} & D_0 \\
                      \end{array}
                    \right] 
\right) \left[
          \begin{array}{c}
            \lambda^{2}x \\
            \lambda x \\
            x \\
           -(D_0-\lambda D_1)^{-1}B x \\
          \end{array}
        \right] = 0.
$$ 
\end{exam}

%\begin{exam}
% Construct one example of symmetric pencil with $A(\lambda)$ of degree 3 and $D(\lambda)$ of degree $2$   
%\end{exam}

\begin{theorem}
Suppose that $G(\lambda)$ is a symmetric rational matrix  and $\mathbb{L}(\lambda) \in \mathbb{L}_{1}(G)$ with right ansatz vector $\left[
\begin{array}{c}
 0 \\
  w \\
   \end{array}
    \right]
$. Then $\mathbb{L}^{T} \in \mathbb{L}_{2}(G)$ with left ansatz vector $\left[
\begin{array}{c}
 0 \\
  w \\
   \end{array}
    \right]$, and $\mathbb{S}(G)\subseteq \mathbb{DL}(G)$.
\end{theorem}

\begin{proof}
Suppose that $\mathbb{L}(\lambda)\in \mathbb{L}_{1}(G)$ with right ansatz vector $\left[
    \begin{array}{c}
              0 \\
              w \\
         \end{array}
          \right] $ and 
%Then $$\left(L(\lambda).\left[
%                  \begin{array}{c}
%                    \Lambda_{m-1} \otimes A(\lambda)^{-1}B  \\
%                    \Lambda_{k-1} \otimes I_m \\
%                  \end{array}
%                \right] 
%                     \right)^{T} = \left[
%                                           \begin{array}{c}
%                                            0 \\
%                                     w \otimes G(\lambda)  \\
 %                                                    \end{array}
%                                              \right]^{T}  $$
%$$\Rightarrow \left[
%                  \begin{array}{c}
%                    \Lambda_{m-1} \otimes A(\lambda)^{-1}B  \\
%                    \Lambda_{k-1} \otimes I_m \\
%                  \end{array}
%                \right]^{T}.L^{T}(\lambda) 
%                         = \left[
%                         \begin{array}{cc}
%                         0 & w^{T} \otimes G(\lambda) \\
%                                             \end{array}
 %                                                \right]$$
%$$ \left[
%        \begin{array}{c}
%         \Lambda_{m-1} \otimes (CA(\lambda)^{-1})^{T}   \\
%                  \Lambda_{k-1} \otimes I_m   \\
%                          \end{array}
%                          \right]^{T}.L(\lambda)
%                                                 = \left[
%                         \begin{array}{cc}
%                         0 & w^{T} \otimes G(\lambda) \\
%                                             \end{array}
%                                                 \right]
%                      $$ as 
$G^{T}(\lambda) = G(\lambda)$. Now, note that  $\mathbb{L}(\lambda)^{T} = \mathbb{L}(\lambda)$. 
Thus by Lemma~\ref{L1L2r}, we get $L^{T}(\lambda) \in \mathbb{L}_{2} (G)$ with left ansatz vector 
$\left[
           \begin{array}{c}
                  0 \\
                  w \\
                  \end{array}
                     \right]$.
Let $\mathbb{L}(\lambda) \in \mathbb{S}(G) \subseteq \mathbb{L}_{1}(G)$ with right ansatz vector $\left[
                         \begin{array}{c}
                                  0 \\
                                 w \\
                            \end{array}
                                  \right]$. 
Then we get $\mathbb{L}^{T}(\lambda) = \mathbb{L}(\lambda)\in \mathbb{L}_{2}(G)$ as $\mathbb{L}(\lambda)$ is symmetric with left ansatz vector 
$\left[
    \begin{array}{c}
               0 \\
                w \\
                \end{array}
                 \right]$. 
Therefore, $\mathbb{L}(\lambda) \in \mathbb{L}_{2}(G)$. Thus $\mathbb{L}(\lambda) \in \mathbb{DL}(G)$. So $\mathbb{S}(G) \subseteq \mathbb{DL}(G)$.
\end{proof}

Now, we define $$ \mathbb{H}(G) := \{\lambda \mathbb{X} + \mathbb{Y} \in \mathbb{L}_{1}(G) : \mathbb{X}^{*} = \mathbb{X}, \,\,\, \mathbb{Y}^{*} = \mathbb{Y} \},$$ the set of all Hermitian pencils in $\mathbb{L}_{1}(G)$. i.e.,
\begin{align*}
&\mathbb{H}(G) = \biggr\{\lambda \left[
                                          \begin{array}{c|c}
                                            X_{11} & 0 \\
                                            \hline
                                            0 & X_{22} \\
                                          \end{array}
                                        \right]+ \left[
                                                   \begin{array}{c|c}
                                                     Y_{11} & e_{m}e_{k}^{T}\otimes C^{*} \\
                                                     \hline
                                                    e_ke_{m}^{T} \otimes C &Y_{22} \\
                                                   \end{array}
                                                 \right] : 
X_{11}^{*} = X_{11}, \,\, X_{22}^{*} = X_{22}, Y_{12} = Y_{21}^{*},\\ & \hspace{1.7cm}  Y_{11}^{*} = Y_{11},   Y_{22}^{*} = Y_{22}, \,\,\, X_{11}\boxplus Y_{11} = e_m \otimes [A_{m} \ldots A_{0}], \,\,\,  X_{22}\boxplus Y_{22} = e_k \otimes [D_{k} \ldots D_{0}]
\biggr\}.
\end{align*}

\begin{theorem}
Suppose that $G(\lambda)$ is a Hermitian rational matrix, and $\mathbb{L}(\lambda) \in \mathbb{H}(G)$ with right ansatz vector $\left[
\begin{array}{c}
 0 \\
  w \\
   \end{array}
    \right]
$. Then $\mathbb{L}(\lambda) \in \mathbb{DL}(G)$ and $\mathbb{H}(G)\subseteq \mathbb{DL}(G)$.
\end{theorem}

\begin{proof}
Since $\mathbb{L}(\lambda)\in \mathbb{L}_{1}(G)$, we have 
$$ \mathbb{L}(\lambda). \left[
                  \begin{array}{c}
                    \Lambda_{m-1} \otimes A(\lambda)^{-1}B  \\
                    \Lambda_{k-1} \otimes I_m \\
                  \end{array}
                \right]  = \left[
       \begin{array}{c}
               0 \\
                w\otimes G(\lambda) \\
              \end{array}
                     \right].$$ 
%Since $G(\lambda)$ and $L(\lambda)$ are Hermitian, we have
%$$\left( L(\lambda).\left[
%                  \begin{array}{c}
%                    \Lambda_{m-1} \otimes A(\lambda)^{-1}B  \\
%                    \Lambda_{k-1} \otimes I_m \\
%                  \end{array}
%                \right] 
%\right)^{*} = \left(\left[
%                      \begin{array}{c}
%                        0 \\
%                        w \otimes G(\lambda) \\
%                      \end{array}
%                    \right]\right)^{*}.$$ 
Since $G^{*} = G(\bar{\lambda})$ and $(\mathbb{L}(\lambda))^{*} = \mathbb{L}(\bar{\lambda})$, we have
$$ \left[
                  \begin{array}{c}
                    \Lambda_{m-1} \otimes A(\lambda)^{-1}B  \\
                    \Lambda_{k-1} \otimes I_m \\
                  \end{array}
                \right] ^{*}. L^{*}(\lambda) = \left[
                      \begin{array}{c}
                       0 \\
                        w^{*}\otimes G^{*}(\lambda) \\
                      \end{array}
                    \right]$$
$$ \Rightarrow \left[
                 \begin{array}{cc}
                   \bar{\Lambda}_{m-1}^{T}\otimes CA(\lambda)^{-1} &  \bar{\Lambda}_{k-1}^{T}\otimes I_m \\
                 \end{array}
               \right] L(\bar{\lambda}) = \left[
                                            \begin{array}{cc}
                                              0 & w^{*} \otimes G(\bar{\lambda})\\
                                            \end{array}
                                          \right].$$ 
This equation holds for all $\lambda$, so we can replace $\bar{\lambda}$ by $\lambda$ to get
$$\left[
    \begin{array}{cc}
\Lambda_{m-1}^{T}\otimes CA(\lambda)^{-1} &   \Lambda_{k-1}^{T}\otimes I_m \\
                 \end{array}
               \right] L(\lambda) = \left[
                                            \begin{array}{cc}
                                              0 & \bar{w}^{T} \otimes G(\lambda) \\
                                            \end{array}
                                          \right].$$
Hence proved.
\end{proof}

Now, we have the following result.

\begin{theorem}
Consider the system matrix $\mathcal{S}(\lambda)$  given in (\ref{sysmatrix01}) with $A(\lambda) =\sum_{i=1}^{m} A_i \lambda^{i}$, $D(\lambda) = \sum_{i=1}^{k} D_i \lambda^{i}$, $m, k >1$. Then the double ansatz pencil 
$\mathbb{L}(\lambda) = \lambda \mathbb{X} + \mathbb{Y}, $ where 
$$\mathbb{X} = \left[
                                \begin{array}{cccc|cccc}
                                   & & & A_{m} & & & &  \\
                                   &  & \rotatebox{40}{$\cdots$} & A_{m-1} &  & & & \\
                                   & \rotatebox{40}{$\cdots$} & \rotatebox{40}{$\cdots$} & \vdots & & & &  \\
                                  A_{m} & A_{m-1} & \cdots& A_{1} & & & & \\
                                  \hline
                                   & & &  & & &  & D_{k}  \\
                                & & &  &  & & \rotatebox{40}{$\cdots$} & D_{k-1} \\
                               & & &   &  & \rotatebox{40}{$\cdots$} & \rotatebox{40}{$\cdots$} & \vdots  \\
                               & & &   &  D_{k} & D_{k-1} & \cdots& D_{1} \\
                                \end{array}
                              \right] $$ $$ Y = \left[
                                \begin{array}{ccccc|ccccc}
                                   & & & -A_{m} & 0 &  & & & &  \\
                                   &  & \rotatebox{40}{$\cdots$} & -A_{m-1} & 0 &  &  & & & \\
                                  & \rotatebox{40}{$\cdots$} &\rotatebox{40}{$\cdots$} & \vdots & \vdots&  & & & & \\
                                  -A_{m} & -A_{m-1} & \cdots& A_{2} & 0 & & & & & \\
                                   0& 0& \cdots& \cdots & A_0 &0  &\cdots &\cdots &0 & B \\
                                  \hline
                                 &  & & &  0& & &  & D_{k} & 0 \\
                              &  & & & \vdots &  &  & \rotatebox{40}{$\cdots$} & D_{k-1} & 0 \\
                              & & & &  \vdots &  & \rotatebox{40}{$\cdots$} & \rotatebox{40}{$\cdots$} & \vdots & 0  \\
                             &  & & & \vdots  &  D_{k} & D_{k-1} & \cdots& D_{1} & \vdots  \\
                             0 & 0& \cdots& \cdots& C & 0 & 0& \cdots& \cdots & D_0   \\
                                \end{array}
                              \right] $$
is block-symmetric with ansatz vector $(e_m, e_k)$. Further, $\mathbb{L}(\lambda)$ is Hermitian (resp., Symmetric) as per $\mathcal{S}(\lambda)$ is Hermitian (resp., Symmetric).

The  pencil $\mathbb{L}(\lambda)$ is a  Rosenbrock linearization of $G(\lambda)$ if and only if  $A_m$ is nonsingular. 
%Equivalently,  $\mathbb{L}(\lambda)$ is a  Rosenbrock  linearization of $G(\lambda)$ if and only if $\mathbb{L}(\lambda)$ is regular.
\end{theorem}

\begin{proof}
Consider the pencil $$\mathbb{L}(\lam) = \left[ \begin{array}{c|c}
           L(\lambda) & -e_me_{k}^{T} \otimes B \\
                              \hline
          e_k e_m^{T} \otimes C & K(\lambda)
                            \end{array} \right], $$ where  $(L(\lambda), e_m) \in\mathcal{DL}(P),$ $ (K(\lambda), e_k) \in\mathcal{DL}(D)$ are the block-symmetric pencils with ansatz vector $e_m$ and $e_k$, respectively. Then note that $\mathbb{L}(\lambda)$ is the block-symmetric pencil with ansatz vector  $(e_m, e_k)$ and $\mathbb{L}(\lambda) \in \mathbb{DL}(G)$ . Also, it is obvious that 
$\mathbb{L}(\lambda)$ is Hermitian (resp., Symmetric) as per $\mathcal{S}(\lambda)$ is Hermitian (resp., Symmetric).

Suppose that  $\mathbb{L}(\lam)$  is a Rosenbrock linearization of $G(\lambda).$  Then there exist $mn\times mn$ unimodular matrix polynomials $U(\lambda)$ and $V(\lambda)$ and $kr\times kr$ unimodular  matrix polynomials $\tilde{U}(\lam)$ and $\tilde{V}(\lam)$ such that
			$$\left[  \begin{array}{c|c} U(\lambda) &  \\ \hline
				&   \tilde{U}(\lam) \end{array} \right]\mathbb{L}(\lambda)\left[  \begin{array}{c|c} V(\lambda) &  \\  \hline   &  \tilde{V}(\lam) \end{array}    \right] = \left[ \begin{array}{c|c} I_{(m-1)n} &   \\    \hline &  \mathcal{S} (\lambda)
\end{array}  \right]$$
for all $\lambda \in \mathbb{C}$.   Hence it follows that $L(\lam)$ is a linearization of $A(\lam) $ and $K(\lam)$ is a linearization of $D(\lam).$ Since $L(\lam) \in \mathcal{DL}(A), K(\lam) \in \mathcal{DL}(D)$ and none of the pencils in $\mathcal{DL}(A), \mathcal{DL}(D)$ is a linearization when $A(\lam), D(\lam)$ is singular \cite[Theorem~6.1]{tds}, the polynomial $A(\lam), D(\lambda)$ must be regular.    
%Now, since $\textbf{p}(x; e_m)$ has a root at $\infty$ and $ L(\lam) \in \mathcal{DL}(P)$ with ansatz vector $e_m$,  by Theorem~\ref{dlpthm} we conclude that $P(\lam)$ does not have an eigenvalue at $\infty.$ This implies that $A_m$ is nonsingular.

%Conversely, suppose that $A_m$ is nonsingular. Then $A(\lam)$ is regular.  Since $ L(\lam) \in \mathcal{DL}(A)$ with ansatz vector $e_m,$  by Theorem~\ref{dlpthm}, $L(\lam)$ is a linearization of $A(\lam).$  Consequently, by Corollary~\ref{lcor}, $\mathbb{L}(\lam)$ is a Rosenbrock linearization of $\mathcal{S}(\lambda).$ Hence by Theorem~\ref{L1_with_v}, $\mathbb{L}(\lam)$ is a Rosenbrock  linearization of $G(\lam).$

%Next, observe that if $A_m$ is singular then  $\mathbb{L}(\lam)$ is a singular pencil. Hence if $\mathbb{L}(\lam)$ is regular then $A_m$ is nonsingular. On the other hand, if $A_m$ is nonsingular then $L(\lam)$ is regular. Since $C A(\lam )^{-1}B $ is strictly proper, it is easily seen that $\det(\mathbb{L}(\lam))\neq 0$ for some $\lam \in \C.$ This shows that $\mathbb{L}(\lam)$ is regular if and only if $A_m$ is nonsingular. 
\end{proof}

\section{Linearizations via non-monomial polynomial bases} 

Although it is customary to consider polynomials in standard monomial basis $\{1,\lambda, \ldots,\lambda^j, \ldots \},$ there are applications in which it is useful to consider other polynomial bases.  Let us consider $\{\phi_0(\lambda), \phi_1(\lambda), \ldots ,\phi_{m-1}(\lambda) \}$ be a basis for the space of scalar polynomials of degree at most $m-1.$ Define

{\footnotesize
$$
\Lambda := \left[\begin{array}{ccccc}
         \lambda^{m-1} & \lambda^{m-2} & \cdots & \lambda & 1 \end{array}\right]^{T} \quad
\mbox{and}\quad 
\Lambda_{\phi}:=\left[\begin{array}{cccc} 
      \phi_0(\lambda) & \phi_1(\lambda) & \cdots & \phi_{m-1}(\lambda \end{array}\right]^{T}.$$
Let $\Phi$ be the unique nonsingular constant matrix such that $\Phi \Lambda = \Lambda_{\phi}.$}

It is shown in \cite{tds} that the vector spaces $\mathcal{L}_1(P)$ and $\mathcal{L}_2(P)$ can be generalized by replacing the monomial basis $\Lambda$  by the  basis $\Lambda_{\phi}$ as follows. Define

\begin{equation*}\label{tils11}
\begin{array}{l}
\widetilde{\mathcal{L}}_1(P) = \left\{\widetilde{L}(\lambda): \,\, \widetilde{L}(\lambda)(\Lambda_{\phi} \otimes I_{n}) = v\otimes P(\lambda), \,\, v \in \mathbb{C}^{m}\right\},\\
\widetilde{\mathcal{L}}_2(P) = \left\{\widetilde{L}(\lambda) : \,\, (\Lambda_{\phi}^{T} \otimes I_{n})\widetilde{L}(\lambda) = w^{T}\otimes P(\lambda),  \,\, w \in \mathbb{C}^{m}\right\},
\end{array}
\end{equation*}
where $\widetilde{L}(\lambda)$ is an $mn \times mn$ matrix pencil. Here $v$ is the right ansatz vector for $\widetilde{L}(\lambda) \in \widetilde{\mathcal{L}}_1(P)$ and $w$ is the left ansatz vector for $\widetilde{L}(\lambda) \in \widetilde{\mathcal{L}}_2(P)$.  Note that if $\widetilde{L} (\lambda) \in \widetilde{\mathcal{L}}_1(P) $ then $L(\lambda):=\widetilde{L} (\lambda)  (\Phi \otimes I_n) \in \mathcal{L}_1(P). $ In fact, it is easily  seen that the map 

$$ \widetilde{\mathcal{L}}_1(P)\longrightarrow \mathcal{L}_1(P), \quad \widetilde{L} (\lambda) \longmapsto \widetilde{L} (\lambda)  (\Phi \otimes I_n)$$ 
is a linear isomorphism~\cite{tds}.

This shows that the pencils in $\mathcal{L}_1(P)$ and $\widetilde{\mathcal{L}}_1(P)$ are strictly equivalent. Hence a pencil $L(\lambda)\in \mathcal{L}_1(P)$ is a linearization of $P(\lambda)$ if and only if the corresponding pencil $\widetilde{L}(\lambda) \in  \widetilde{\mathcal{L}}_1(P)$ is a linearization of $P(\lambda).$ 

Consider the system matrix $\mathcal{S}(\lambda)$ and 
the transfer function as defined in (\ref{sysmatrix01}) and (\ref{tfunction01}). Let

{\footnotesize
\begin{align*}
\Lambda_{m-1} &:= \left[\begin{array}{ccccc}
         \lambda^{m-1} & \lambda^{m-2} & \cdots & \lambda & 1 \end{array}\right]^{T} \quad
\mbox{and}\quad 
\Lambda_{\phi}&:=\left[\begin{array}{cccc} 
      \phi_0(\lambda) & \phi_1(\lambda) & \cdots & \phi_{m-1}(\lambda \end{array}\right]^{T}\\
\Lambda_{k-1} &:= \left[\begin{array}{ccccc}
         \lambda^{k-1} & \lambda^{k-2} & \cdots & \lambda & 1 \end{array}\right]^{T} \quad
\mbox{and}\quad 
\Lambda_{\psi}&:=\left[\begin{array}{cccc} 
      \phi_0(\lambda) & \phi_1(\lambda) & \cdots & \phi_{m-1}(\lambda \end{array}\right]^{T}.
\end{align*}}
Let $\Phi$ be the unique nonsingular constant matrix such that $\Phi\Lambda_{m-1}=\Lambda_{\phi}$ and $\Psi$ be the unique nonsingular matrix such that $\Psi\Lambda_{k-1}=\Lambda_{\psi}$. We know that if $\widetilde{L}(\lambda)\in\widetilde{\mathcal{L}}_{1}(A)$ then $L(\lambda)=\widetilde{L}(\lambda)\left(\Phi\otimes I_{n}\right) \in \mathcal{L}_{1}(A)$ and if $\widetilde{K}(\lambda)\in\widetilde{\mathcal{L}}_{1}(A)$ then $K(\lambda)=\widetilde{K}(\lambda)\left(\Psi\otimes I_{r}\right) \in \mathcal{L}_{1}(D)$. We define $\widetilde{\mathbb{L}}_1(G)$ and $\widetilde{\mathbb{L}}_2(G)$ for the system matrix $\mathcal{S}(\lambda)$ and the transfer function $G(\lambda)$ as follows:

{\small
\beano
 \widetilde{\mathbb{L}}_{1}(G) &:=&
    \bigg\{\left[\begin{array}{c|c}
       \widetilde{L}(\lambda) & -v e_{k}^{T}\Psi^{-1} \otimes B  \\\hline
       we_{m}^{T} \Phi^{-1} \otimes C & \widetilde{K}(\lambda) \\
    \end{array}\right]  :
 (\widetilde{L}(\lambda) , v) \in \widetilde{\mathcal{L}}_1(A),  (\widetilde{K}(\lambda) , w) \in \widetilde{\mathcal{L}}_1(D)  \bigg\}, \\
\widetilde{\mathbb{L}}_{2}(G) &:= & 
    \bigg\{\left[\begin{array}{c|c}
       \widetilde{L}(\lambda) & -\Psi^{-T}e_k s^T \otimes B\\\hline
       \Phi^{-T}e_{m}z^{T} \otimes C & \widetilde{K}(\lambda)
    \end{array}\right]:   
  ( \widetilde{L}(\lambda), s)\in \widetilde{\mathcal{L}}_2(A), (\widetilde{K}(\lambda), z) \in \widetilde{\mathcal{L}}_2(D) \bigg\}.
 \eeano}

\begin{theorem}
Let 
{\small $\widetilde{\mathbb{T}}(\lambda)
    =\left[\begin{array}{c|c}
        \widetilde{L}(\lambda) & -v e_{k}^{T}\Psi^{-1} \otimes B  \\\hline
        we_{m}^{T} \Phi^{-1} \otimes C & \widetilde{K}(\lambda) \\
     \end{array}\right] 
    \in \widetilde{\mathbb{L}}_{1}(G)$},
where 
{\small$(\widetilde{L}(\lambda), v) \in \widetilde{\mathcal{L}}_1(A)$} and 
{\small$(\widetilde{K}(\lambda), w) \in \widetilde{\mathcal{L}}_1(D)$} with  $v, w\neq 0.$ Then 
{\small $\mathbb{T}(\lambda)
       =\widetilde{\mathbb{T}}(\lambda)
           \left[\begin{array}{c|c}
              \Phi \otimes I_n &  \\\hline
              &  \Psi \otimes I_r 
        \end{array} \right] 
        \in \mathbb{L}_{1}(G)$}.
In fact, the map   
{\small$\mathbb{J}: \widetilde{\mathbb{L}}_1(G)\longrightarrow \mathbb{L}_1(G),~\widetilde{\mathbb{T}}(\lambda) \longmapsto \widetilde{\mathbb{T}}(\lambda) \left[\begin{array}{c|c}
           \Phi \otimes I_n&  \\
           \hline
            &  \Psi \otimes I_r \end{array} \right] $}
is a linear isomorphism.
\end{theorem}

\begin{proof} 
Note that $ L(\lam) = \widetilde{L} (\lambda)  (\Phi \otimes I_n) \in \mathcal{L}_1(A)$ with right ansatz vector $v$ and $ K(\lam) = \widetilde{K}(\lambda)(\Psi \otimes I_r) \in \mathcal{L}_1(D)$ with right ansatz vector $w$. Then,
 \begin{align*}
 \widetilde{\mathbb{T}}(\lambda)
     \left[\begin{array}{c|c}
           \Phi \otimes I_n &  \\
           \hline
            &  \Psi \otimes I_r 
     \end{array} \right] 
  = \left[\begin{array}{c|c}
        \widetilde{L}(\lambda) (\Phi \otimes I_n) & -v e_{k}^{T} \otimes B  \\\hline
        we_{m}^{T}  \otimes C & \widetilde{K}(\lambda) (\Phi \otimes I_r) \\
     \end{array}\right]  \in \mathbb{L}_{1}(G).
\end{align*}
Obviously $\mathbb{J} : \widetilde{\mathbb{L}}_1(G)\longrightarrow \mathbb{L}_1(G)$ is linear and injective.
Since 
$\mathbb{T}(\lambda)
   \left[\begin{array}{c|c}
           \Phi^{-1} \otimes I_n&  \\
           \hline
            &  \Psi^{-1} \otimes I_r \end{array} \right] \in \widetilde{\mathbb{L}}_1(G)$  
and $\mathbb{J}$ maps it to $\mathbb{T}(\lambda)$ for all $\mathbb{T}(\lambda) \in \mathbb{L}_1(G),$ we conclude that $\mathbb{J}$ is onto. This completes the proof.
\end{proof}

Analogously, the map   
$\mathbb{J}: \widetilde{\mathbb{L}}_2(G)\rightarrow \mathbb{L}_2(G),~\widetilde{\mathbb{T}}(\lambda) \mapsto \left[\begin{array}{c|c}
           \Phi^T \otimes I_n&  \\
           \hline
            & \Psi^T \otimes I_r \end{array} \right] \widetilde{\mathbb{T}}(\lambda)  $
is a linear isomorphism.

Observe that 
$ \widetilde{\mathbb{T}}(\lambda) \in \widetilde{\mathbb{L}}_1(G)$ is a Rosenbrock  linearization of $G(\lambda)$ if and only if  $\mathbb{T}(\lam):=\widetilde{\mathbb{T}}(\lambda)  
    \left[\begin{array}{c|c}
           \Phi \otimes I_n&  \\
           \hline
            & \Psi \otimes I_r 
    \end{array} \right] \in \mathbb{L}_1(G)$
is a Rosenbrock linearization of $G(\lambda).$ Similarly,
$ \widetilde{\mathbb{H}}(\lambda) \in \widetilde{\mathbb{L}}_2(G)$ is a Rosenbrock linearization of $G(\lambda)$ if and only if  $\mathbb{H}(\lam) :=  
    \left[\begin{array}{c|c}
             \Phi^T \otimes I_n&  \\
             \hline
             & \Psi^T \otimes I_r 
    \end{array} \right] 
    \widetilde{\mathbb{H}}(\lambda)   \in \mathbb{L}_2(G)$ 
is a Rosenbrock linearization of $G(\lambda).$  

\end{document}